\documentclass[12pt, reqno]{amsart}
\usepackage{mathrsfs}
\usepackage[active]{srcltx}
\usepackage{mathrsfs,amsmath}
\usepackage{mathtools}
\usepackage{longtable}
\usepackage{todonotes}
\usepackage{amssymb}
\usepackage{cite}
\usepackage{moreenum}
\usepackage{float}
\usepackage{enumitem}
\usepackage{dsfont}

\setlist[description]{font=\normalfont\bfseries, labelsep=0.5em}

\allowdisplaybreaks

\usepackage{xcolor}
\definecolor{bluecite}{HTML}{0875b7}

\usepackage[unicode=true,
bookmarksopen={true},
pdffitwindow=true,
colorlinks=true,
linkcolor=bluecite,
citecolor=bluecite,
urlcolor=bluecite,
hyperfootnotes=false,
pdfstartview={FitH},
pdfpagemode= UseNone]{hyperref}

\newcommand{\ds}{\displaystyle}

\newcommand{\be}{\begin{equation}}
\newcommand{\ee}{\end{equation}}

\newcommand{\Z}{\mathbb{Z}}

\newcommand{\R}{\mathbb{R}}
\newcommand{\CC}{\mathbb{C}}
\newcommand{\N}{\mathbb{N}}

\usepackage[left=2.0cm,right=2.0cm,top=2.0cm,bottom=2.0cm]{geometry}

\newtheorem{example}{Example}[section]
\newtheorem{proposition}{Proposition}[section]
\newtheorem{theorem}{Theorem}[section]
\newtheorem{lemma}{Lemma}[section]

\newtheorem{remark}{Remark}[section]

\numberwithin{equation}{section}

\usepackage{bbm}  
\usepackage{dsfont}


\title[Quantitative stability of extremal quasiconformal mappings]{Quantitative stability of extremal quasiconformal mappings}

\author[Zolt\'an M. Balogh]{Zolt\'an M. Balogh}
\address{Zolt\'an M. Balogh, Universit\"at Bern\\ Mathematisches Institut (MAI)\\ Sidlerstrasse 12\\ 3012 Bern\\ Schweiz}
\email{zoltan.balogh@unibe.ch}

\author[K\'aroly J. B\"or\"oczky]{K\'aroly J. B\"or\"oczky}
\address{K\'aroly J. B\"or\"oczky, Alfr\'ed R\'enyi Institute of Mathematics, 
 Realtanoda u. 13-15, 1053 Budapest, Hungary \& ELTE, Institute of Mathematics, P\'azm\'any P\'eter s\'et\'any 1/C, 1117 Budapest, Hungary}
 \email{boroczky.karoly.j@renyi.hu}

\author[\'Agnes Mester]{\'Agnes Mester}
\address{\'Agnes Mester, 
Universit\"at Bern\\Mathematisches Institut (MAI)\\	Sidlerstrasse 5\\	3012 Bern\\ Switzerland 
\& Babe\c s-Bolyai University, Faculty of Mathematics and Computer Science, Mihail Kogalniceanu Str. 1,
	400084 Cluj-Napoca, Romania
}
\email{agnes.mester@unibe.ch; agnes.mester@ubbcluj.ro}

\thanks{Z. M. Balogh and \'A. Mester are supported by the Swiss National Science Foundation, Grant Nr. {200020\_228012}. K. J. B\"or\"oczky is supported by the NKFIH grant 150613.  }
 \thanks{}
\begin{document}
\begin{abstract}
We establish quantitative stability results for classical distortion minimization problems in the theory of quasiconformal mappings. 
We consider the mean distortion functional and prove sharp stability estimates for the minimization problems regarding the linear stretch and spiral stretch maps, which arise as extremals in the class of mappins with finite distortion under appropriate boundary conditions. 
More precisely, we show that if a mapping has mean distortion close to the minimal value in the appropriate function class, then it must be quantitatively close, in certain $L^p$-norms for $p\geq 1$, to the corresponding extremal map.

 \end{abstract}

\maketitle

\section{Introduction and main results}

Quantitative stability of geometric and functional inequalities has become a major theme in analysis in recent years. Beyond proving sharpness and identifying extremizers, such results focus on quantifying how the deficit from equality controls the distance to the set of extremizers. Roughly speaking, if a function or set "almost attains" equality in these sharp inequalities, then it must be quantitatively close (measured in a suitable sense) to the family of extremizers. Main results in this direction include the sharp quantitative isoperimetric inequality of Fusco, Maggi, and Pratelli \cite{FMP08}, stability for the Brunn-Minkowski inequality due to Figalli, Maggi, and Pratelli \cite{FMP}, and quantitative versions of the Pr\'ekopa-Leindler inequality obtained by Ball and B\"or\"oczky \cite{BB}, B\"or\"oczky and De \cite{BD}, and Figalli, van Hintum, and Tiba \cite{FvHT}. Further important contributions concern sharp stability results for the Sobolev inequality, see Bianchi and Egnell \cite{BianchiEgnell} for $p=2$, Figalli, Maggi, and Pratelli \cite{FMP_2013} for $p=1$, and Figalli and Zhang \cite{FZ} for $1 < p <n$. For comprehensive surveys on stability estimates, see Figalli \cite{Figalli}, Fusco \cite{Fusco}, or Frank \cite{Frank}.

A prototypical example is provided by the  stability of the sharp Euclidean isoperimetric inequality. The classical result states that for any open smooth bounded set $E \subset \mathbb R^n$, one has
\begin{equation}\label{eq:isop_ineq}
    \mathcal P(E) \geq n \omega_n^\frac{1}{n} |E|^\frac{n-1}{n},
\end{equation}
where $\omega_n$ denotes the volume of the $n$-dimensional unit ball in $\mathbb R^n$, while $\mathcal P(E)$ and $|E|$ are the perimeter and  volume of the set $E$. 
The sharp quantitative stability version due to Fusco, Maggi, and Pratelli \cite{FMP08} states that there exists a dimensional constant $C = C(n)>0$  such that if 
\begin{equation}\label{eq:almost_eq_intro}
    \mathcal P(E) \leq (1+\varepsilon) n \omega_n^\frac{1}{n} |E|^\frac{n-1}{n}
\end{equation}
for some $\varepsilon >0$,
then there exists a point $x \in \mathbb R^n$ such that
\begin{equation} \label{eq:quantitative_isoperimetric}
    \frac{1}{|E|} \int_{\mathbb R^n} |\mathds{1}_E - \mathds{1}_{B_r(x)} |(y) dy \leq C \sqrt{\varepsilon},
\end{equation}
where $B_r(x)$ denotes the $n$-dimensional Euclidean ball with center $x$ and radius $r$, and $r>0$ is given by the relation $r^n \omega_n = |E|$.
Moreover, the rate $\sqrt{\varepsilon}$ is sharp.

Motivated by this perspective, the purpose of this paper is to establish similar quantitative stability results in the context of distortion-minimizing quasiconformal mappings. More precisely, we revisit classical extremal problems for the mean distortion functional and prove sharp stability estimates:  we show that if a mapping has mean distortion close to the minimal value in the admissible function class, then it must be quantitatively close (in some suitable $L^p$-norm) to the corresponding extremal quasiconformal map.

Let us recall briefly the relevant extremal problems. Let $\Omega$ and $\Omega'$ be bounded domains in the complex plane $\CC$, and let $f: \overline{\Omega} \to \overline{\Omega'} $ be an orientation-preserving mapping. Then $f$ is said to have finite distortion if $f \in W^{1,1}(\Omega, \Omega')$, $J(\cdot,f) \in L^1(\Omega)$ and there exists a measurable function $K: \Omega \to \mathbb R \cup \{\infty\}$ such that
$$1\leq K(z) < \infty \  \ \text{and} \ \  |Df(z)|^2 \leq K(z) J(z,f), \ \text{for a.e.} \ z\in \Omega,$$
where 
$|Df(z)|= |f_z(z)| + |f_{\bar{z}}(z)|$
is the norm of the $\R$-linear differential map $Df(z): \R^2 \to \R^2$ and $J(z,f)= |f_z(z)|^2 - |f_{\bar{z}}(z)|^2\geq 0$ is its Jacobian determinant at a.e. $z\in \Omega$, see Astala, Iwaniec, and Martin \cite{AIM09}. 

The linear distortion of $f$ is defined for a.e. $z \in \Omega$ as 
$$
K(z,f)=
\left\{\begin{array}{rl}
\ds\frac{|f_z(z)|+|f_{\bar{z}}(z)|}{|f_z(z)|-|f_{\bar{z}}(z)|}, & \mbox{ if }|f_{\bar{z}}(z)|< |f_z(z)|\\[1ex]
1,&\mbox{ otherwise.}
\end{array}\right.,
$$
while the associated mean distortion functional is given by 
$$ f \mapsto \int_{\Omega} \varphi(K(z,f)) \rho(z) d\mathcal{L}^2(z),$$
where  $\varphi: [1, \infty) \to \R$ is an increasing strictly convex function with $\varphi(1) = 1$ and $\rho: \Omega \to \R_+$ is a given density. 
A central problem in the theory of extremal quasiconformal mappings is to minimize the above functional over an admissible class $\mathcal{F} \subset W^{1,1}(\Omega, \Omega')$ of finite distortion mappings subject to certain boundary conditions. 

A classical example is the Gr\"otzsch-type \cite{Gr} minimization problem between rectangles, which has as solution the linear stretch map. Let $k,\ell>0$ and $n\in\R$ be fixed, and consider the rectangle
$$
Q_1=\{z=x+iy \in \mathbb C:x\in[0,\ell] \text{ and }y\in[0,1]\}.
$$ 
Let $\mathcal{F}$ be the set of all orientation-preserving homeomorphisms $f \in W^{1,1}(Q_1)$ of finite distortion, satisfying the boundary conditions $f(0) = 0$,
\begin{align*}
f(x+i)&=f(x)+i,&&x\in[0,\ell],\\
f(\ell+iy)&=f(iy)+k\ell+in\ell,&&y\in[0,1].
\end{align*}
Then, it is straightforward to verify that the linear stretch map $f^*:Q_1\to \CC$,   
\begin{equation} \label{linear_stretch_map}
    f^*(z)=kx+inx+iy, \quad z=x+iy
\end{equation}
belongs to $\mathcal{F}$.
Moreover, due to the result of Feng, Hu, and Shen \cite[Theorem 2]{FHS}, if $\varphi: [1, \infty) \to (0, \infty)$ is an increasing convex function, then for every $f\in \mathcal{F}$,  
\begin{equation}
\label{intro_minimizer1}
\int_{Q_1} \varphi(K(z,f))\,d\mathcal{L}^2(z)\geq \int_{Q_1} \varphi(K(z,f^*))\,d\mathcal{L}^2(z),
\end{equation}
i.e., $f^{\ast}$ minimizes the mean distortion functional over the class $\mathcal{F}$.
Furthermore, if $\varphi$ is strictly convex, then equality holds  if and only if $f= f^*$. As we show in Example~\ref{ex:non-unique},  strict convexity of $\varphi$ is indeed a necessary assumption for the uniqueness of the minimizer.

By using exponential and logarithmic changes of coordinates, the above result yields the solution of the corresponding minimization problem between two annuli, see, e.g., Balogh, F\"assler, and Platis \cite{BFP}, Feng, Hu, and Shen \cite{FHS}, Feng and Zhang \cite{FZ}, and Gutlyanskii and Martio \cite{GM}. 
More precisely, let  
\begin{equation} \label{annuli}
A_1 = \{ w\in \CC: q \leq |w| \leq 1 \} \ \text{and} \ A_2 = \{ w\in \CC: q^k \leq |w| \leq 1 \}
\end{equation}
be two annuli, where $0<q<1$ and $k>0$. 
Then, for $\theta \in [-\pi, \pi]$, the extremal spiral-stretch map for the mean distortion functional is  given by  $g^{\ast}: A_1 \to A_2$, 
\begin{equation} \label{eq:spiral-stretch}
g^{\ast}(w)= w|w|^{k-1}\exp\left(i\frac{\theta\log|w|}{\log q}\right) .
\end{equation}
Observe that  $g^{\ast}$ fixes the outer boundary of $A_1$, while the inner boundary is stretched by a factor $k$ and is rotated by an angle $\theta$,  transforming radial lines into logarithmic spirals. Spiral stretches have important applications in several contexts, for instance in the work of Gehring \cite{Ge} on the universal Teichm\"uller space or in John's studies \cite{J}, \cite{J2} of the nonlinear elastic equilibrium with prescribed boundary displacements. The spiral stretch map was generalized to the sub-Riemannian setting of the Heisenberg group by Balogh, F\"assler and Platis \cite{BFP2}. Due to the result of Gutlyanskii and Martio \cite{GM} (see also Balogh, F\"assler and Platis \cite{BFP}), it turns out that $g^{\ast}$ minimizes the mean distortion functional in the class of homeomorphic
$W^{1,2}-$mappings of finite distortion. 
This result was generalized to the space $W^{1,1}(A_1)$ by Feng, Hu and Shen \cite{FHS} (see also Feng and Zhang \cite{FZ}), who also proved that $g^{\ast}$ is actually the unique minimizer of the mean distortion functional. Namely, if $\varphi:[1,\infty)\to[1,\infty)$ is increasing and strictly convex with $\varphi(1)=1$, then for any orientation-preserving, finite-distortion homeomorphism $g:A_1\to A_2$  satisfying $g=g^{\ast}$ on $\partial A_1$, one has 
\begin{equation}
\label{intro_minimizer2}
\int_{A_1} \frac{\varphi(K(w,g))}{|w|^2}\,d\mathcal{L}^2(w)\geq 
\int_{A_1} \frac{\varphi(K(w,g^{\ast}))}{|w|^2}\,d\mathcal{L}^2(w),
\end{equation}
with equality if and only if $g=g^{\ast}$.
Again, the  strict convexity of $\varphi$ is essential for uniqueness, see Example~\ref{ex:stretch-sharp}.

The main contribution of the present paper is to complement the above minimization results with quantitative stability estimates, in the spirit of the quantitative theorem of 
Fusco, Maggi, and Pratelli \cite{FMP08} (see relations \eqref{eq:almost_eq_intro} and \eqref{eq:quantitative_isoperimetric}). Specifically, we establish sharp stability versions of the inequalities \eqref{intro_minimizer1} and \eqref{intro_minimizer2}, showing that if a homeomorphism from the admissible function class has nearly minimal mean distortion, then it must be quantitatively close in the $L^1$ norm to the corresponding extremal mapping -- namely, the linear stretch in the case of \eqref{intro_minimizer1} or the spiral stretch in case of \eqref{intro_minimizer2}.
We also extend our results to different $L^p$ stability estimates with $p\geq 1$.

Throughout the paper, we shall use the following notation. For positive quantities $A$ and $B$, we write $A\ll B$ (equivalently, $B\gg A$) if there exists a constant $C>0$ such that $A\leq C\cdot B$. In the  Gr\"otzsch-type minimization problem \eqref{intro_minimizer1}, the constant $C$ depends only on  $k, \ell, n, \varphi$ (and additionally on $p >1$ in the case of the $L^p$ estimates). In the minimization problem  \eqref{intro_minimizer2}, $C$ may depend only on $k$,$q$, and $\varphi$ (and again, on $p$ in the case of the $L^p$ stability estimates).

Using this convention, our first result reads as follows.

\begin{theorem} \label{Thm:Lin-Stretch}
Let $\varphi:[1, \infty) \to [1, \infty)$ be an increasing and strictly convex function satisfying $\varphi(1) = 1$ and $\varphi''(t) > c_0$ for some constant $c_0 > 0$ and for a.e. $t\in [1, \infty)$. Suppose that for some $f \in \mathcal{F}$ and some $\varepsilon>0$ small enough, one has
\begin{equation*}
\int_{Q_1} \varphi(K(z,f))\,d\mathcal{L}^2(z)\leq (1+\varepsilon)\int_{Q_1} \varphi(K(z,f^*))\,d\mathcal{L}^2(z),
\end{equation*}
and 
\begin{equation} \label{eq:boundary-cond}
    \int_{\partial Q_1}|f-f^{\ast}|(z) \,dz \ll \sqrt{\varepsilon}.
\end{equation}
Then
\begin{equation*} 
\int_{Q_1}|f-f^{\ast}|(z) \,d\mathcal{L}^2(z) \ll \sqrt{\varepsilon}.
\end{equation*}
Moreover, the exponent $1/2$ in the factor $\sqrt{\varepsilon}$  is sharp. 
\end{theorem}

Furthermore, the above result can be extended to $L^p$  stability estimates with $p\geq 1$ if we additionally assume that $f=f^{\ast}$ on $\partial Q_1$. Specifically, if $1\leq p<2$, then we have the stronger estimate 
\begin{equation*} 
\left(\int_{Q_1}|f-f^{\ast}|^p(z) \,d\mathcal{L}^2(z)\right)^{1/p} \ll \sqrt{\varepsilon},
\end{equation*} 
while if $p\geq 2$, we have the weaker result
\begin{equation*} 
\left(\int_{Q_1}|f-f^{\ast}|^p(z) \,d\mathcal{L}^2(z)\right)^\frac{1}{p} \ll \varepsilon^\frac{1}{2p}.
\end{equation*}

Next, we consider the distorsion minimization problem between the annuli and prove the sharp stability version of \eqref{intro_minimizer2}. Again, we assume uniform convexity on the function $\varphi$.

\begin{theorem} \label{Thm:Spiral-Stretch-stab}
Let $0<q<1$, $k>0$, $0<\theta < 2\pi$, and let $\varphi: [1, \infty) \to [1, \infty)$ be an increasing and strictly convex function satisfying $\varphi(1)=1$ and $\varphi''(t) > c_0$ for some constant $c_0>0$ and for a.e. $t\in [1, \infty)$. Then there exists $\varepsilon_0 >0$ such that if $g:A_1\to A_2$ is an orientation preserving homeomorphism  with finite distortion in $W^{1,1}(A_1, A_2)$, 
$g=g^{\ast}$ on $\partial A_1$, and
\begin{equation*}
\int_{A_1} \frac{\varphi(K(w,g))}{|w|^2}\,d\mathcal{L}^2(w)\leq(1+\varepsilon) \int_{A_1} \frac{\varphi(K(w,g^{\ast}))}{|w|^2}\,d\mathcal{L}^2(w)
\end{equation*}
holds for $0< \varepsilon < \varepsilon_0$, then 
\begin{equation*}
\int_{A_1} |g-g^{\ast}|(z)\,d\mathcal{L}^2(z) \ll \sqrt{\varepsilon}.
\end{equation*}
Moreover, the exponent $1/2$ is sharp. 
\end{theorem}

Similarly to the Gr\"otzsch problem, we also establish $L^p$-type stability estimates with factor $\sqrt{\varepsilon}$ in the case $1\leq p <2$, and $\varepsilon^\frac{1}{2p}$ when $ p \geq 2$. 

Note that in Theorems \ref{Thm:Lin-Stretch} and \ref{Thm:Spiral-Stretch-stab}, the uniform convexity assumption plays an indispensable role, as it allows us to use a second-order Taylor expansion of the distortion functional and ultimately yields the sharp $\sqrt{\varepsilon}$ rate. In fact, in Example \ref{ex:unif-convexity-assumption} we show the sharpness of the assumed uniform convexity condition. 

Finally, we turn to the general setting where $\varphi$ is assumed to satisfy merely the necessary strict convexity condition.
In this case, the above Taylor-type  argument no longer applies; instead, we need to introduce some carefully constructed Young functions associated to $\varphi$ in order to apply the Orlicz-Hölder inequality. This way we are able to derive weaker -- but still  optimal -- stability results.
For example, in the case of the minimization problem between the two annuli, we have the following theorem:

\begin{theorem} 
\label{Thm:Spiral_Stretch-weak-Stab}
Let $0<q<1$, $k>0$ and $\theta \in [-\pi, \pi]$. Let $K_0>K(w,g^{\ast}),$ for all $w \in A_1$, where $g^{\ast}: A_1\to A_2$ is the spiral-stretch map as in \eqref{eq:spiral-stretch} and $A_1$, $A_2$ are the annuli as in \eqref{annuli}. Let $\varphi: [1, \infty) \to [1, \infty)$ be an increasing and strictly convex function satisfying $\varphi(1)=1$. Then there exist $\varepsilon_0 >0$ and a function $\delta:[0,\varepsilon_0]\to[0,\infty)$ depending on $q,k,\theta,K_0,\varphi$, and satisfying $\delta(\varepsilon)>0$ for all $\varepsilon\in(0,\varepsilon_0)$ and $\lim_{\varepsilon\to 0^+}\delta(\varepsilon)=0$, with the following property. If $g:A_1\to A_2$ is a quasiconformal mapping satisfying  $K(w,g)\leq K_0$ for all $w\in A_1$,
and $g=g^{\ast}$ on $\partial A_1$, such that
\begin{equation*}
\int_{A_1} \frac{\varphi(K(w,g))}{|w|^2}\,d\mathcal{L}^2(w)\leq(1+\varepsilon) 
\int_{A_1} \frac{\varphi(K(w,g^{\ast}))}{|w|^2}\,d\mathcal{L}^2(w)
\end{equation*}
holds for  $0< \varepsilon < \varepsilon_0$, then 
\begin{equation*}
\int_{A_1} |g-g^{\ast}|(z) \,d\mathcal{L}^2(z)\leq \delta(\varepsilon).
\end{equation*}
\end{theorem}

A similar weak stability estimate holds for the linear stretch map, see Theorem \ref{Thm:Linear_Stretch-Stab-weak}.
We also construct an example 
which shows that Theorem \ref{Thm:Spiral_Stretch-weak-Stab} is in fact optimal, see Example~\ref{ex:strict-conv-error}.

The paper is organized as follows. In Section \ref{secS2}  we focus on the  Gr\"otzsch-type minimization problem for rectangles, proving Theorem \ref{Thm:Lin-Stretch}. In Section  \ref{secS3}, we use Theorem \ref{Thm:Lin-Stretch} together with exponential and logarithmic coordinate changes to show Theorem~\ref{Thm:Spiral-Stretch-stab}. 
In Section \ref{Sec:Weak} we elaborate how the arguments of Section \ref{secS2} and \ref{secS3} can be adapted to the strictly convex case, obtaining Theorem  \ref{Thm:Spiral_Stretch-weak-Stab} and the corresponding weaker stability result for the linear stretch map (see Theorem \ref{Thm:Linear_Stretch-Stab-weak}). Section \ref{secS5} concludes with generalizations to $L^p$ stability estimates for $p\geq 1$, as well as some remarks and open questions. Finally,  we include the necessary background regarding Young functions in the Appendix.

\section{Quantitative stability for the linear stretch map}
\label{secS2}

In this section, we consider the Gr\"otzsch-type minimization problem of the mean distortion of maps between quadrilaterals, following the framework of  Feng, Hu, and Shen \cite{FHS}. 
Let $k,\ell>0$ and $n\in\R$ be fixed, and consider the rectangle
$$
Q_1=\{z=x+iy \in \mathbb C:x\in[0,\ell] \text{ and }y\in[0,1]\}
$$
and the associated lattice
$L=\Z i+\Z(k\ell+in\ell)$.
Let  $\mathcal{F}$ be the class of orientation-preserving homeomorphisms $f\in W^{1,1}(Q_1)$  with finite distortion satisfying the following boundary conditions:
\begin{align}
\label{condf0}
f(0)&=0, &&\\
\label{condfx}
f(x+i)&=f(x)+i,&&x\in[0,\ell],\\
\label{condfy}
f(\ell+iy)&=f(iy)+k\ell+in\ell,&&y\in[0,1].
\end{align}
In particular, these assumptions imply that $f(Q_1)$ is a fundamental domain for $L$. In addition, note that every homeomorphisms $f\in W^{1,1}(Q_1)$ automatically satisfies the regularity condition $J(\cdot,f) \in L^1(\Omega)$. 

Note that the linear stretch map $f^*:Q_1 \to f^*(Q_1)$, 
$$
f^*(z)=kx+inx+iy,
$$
for $z=x+iy$, $x,y\in\R$, is an element of the class $\mathcal{F}$. In the sequel we use the notation $f^*(Q_1) = Q_2$.
The complex dilatation of $f^*$ is constant and can be computed directly as 
\begin{equation*}
\mu^*=\frac{f^*_{\bar{z}}}{f^*_z}= \frac{k-1+in}{k+1 +in},\mbox{ \ \ with }0<|\mu^*|<1.
\end{equation*}
Consequently, the linear distortion of $f^*$ is constant for all $z$ and is given by
\begin{equation*}
    K(z,f^*)=\frac{1+|\mu^*|}{1-|\mu^*|}=\frac{|f^*_z|+|f^*_{\bar{z}}|}{|f^*_z|-|f^*_{\bar{z}}|}= \frac{\sqrt{(k+1)^2 + n^2} + \sqrt{(k-1)^2 + n^2}}
{\sqrt{(k+1)^2 + n^2} - \sqrt{(k-1)^2 + n^2}}.
\end{equation*}

According to Feng, Hu, and Shen \cite[Theorem 2]{FHS}, if $\varphi: [1, \infty) \to (0, \infty)$ is an increasing convex function, then for any  homeomorphism $f\in \mathcal{F}$, one has that 
\begin{equation}
\label{Inequality}
\int_{Q_1} \varphi(K(z,f))\,d\mathcal{L}^2(z)\geq \int_{Q_1} \varphi(K(z,f^*))\,d\mathcal{L}^2(z),
\end{equation}
i.e., the linear stretch map $f^{\ast}$ minimizes the mean distortion functional over the class $\mathcal{F}$.
Moreover, if $\varphi$ is strictly convex, then equality holds in \eqref{Inequality} if and only if $f= f^*$.

The following example shows that the strict convexity of $\varphi$ is indeed a necessary assumption for the uniqueness of the extremal map. Indeed, in the case when $\varphi(t) = t$, although \eqref{Inequality} holds,  the mean distortion functional has infinitely many minimizers.

\begin{example} \label{ex:non-unique}
Let $\ell=1$, $k>1$, $Q_1=[0,1]\times[0,1]$,  $Q_2 = [0,k]\times [0,1]$, and consider the linear stretch $f^{\ast}: Q_1 \to Q_2$ given by $f^{\ast}(x+iy) = kx +iy$. Then, for any $0< \varepsilon < (k-1)^2$, there exists a map $f_{\varepsilon} \in \mathcal{F}$ such that $f_{\varepsilon} \neq f_{\varepsilon'}$ for every $\varepsilon\neq \varepsilon'$, and $f_{\varepsilon}$ satisfies 
\begin{equation*}
 \int_{Q_1} K(z, f_{\varepsilon}) d\mathcal{L}^2 =\int_{Q_1} K(z, f^{\ast}) d\mathcal{L}^2.     
\end{equation*}
\end{example}

\begin{proof}
We shall give an explicit formula for the maps $f_\varepsilon \in \mathcal{F}$ with the properties stated in the example. 
First, observe that 
$K(z, f^{\ast}) = k $ is constant.
Let $0< \varepsilon < (k-1)^2$ be fixed, and consider the map $f_{\varepsilon}: Q_1 \to Q_2 $, $f_{\varepsilon}(x+iy) = g_{\varepsilon}(x) +iy$, where $g_{\varepsilon}:[0,1] \to [0,k]$
is the piecewise linear map
$$
g_{\varepsilon}(x)=
\left\{\begin{array}{rl}
(k+ \sqrt{\varepsilon})x, & \mbox{ if } \ x\in [0,\frac{1}{2}], \\ (k-\sqrt{\varepsilon})x + \sqrt{\varepsilon}, 
& \mbox{ if } \ x \in [\frac{1}{2}, 1].
\end{array}\right.
$$
For any $z= x +iy$ with $x\neq \frac{1}{2}$, we have 
\begin{equation*}
K(z, f_{\varepsilon}) = g_{\varepsilon}'(x)=
\left\{\begin{array}{rl}
k+ \sqrt{\varepsilon} & \mbox{ if } \ x\in (0,\frac{1}{2}), \\ k-\sqrt{\varepsilon} 
&\mbox{ if } \ x \in (\frac{1}{2}, 1), 
\end{array}\right.
\end{equation*}
 which in turn yields that 
 $$ 
\int_{Q_1} K(z, f_{\varepsilon}) d\mathcal{L}^2 =\int_{Q_1} K(z, f^{\ast}) d\mathcal{L}^2.
$$
This example will be used repeatedly in the sequel. 
\end{proof}

The proof of our first main result, Theorem \ref{Thm:Lin-Stretch}, is based on the following proposition:

\begin{proposition}
\label{prop:lin-quant} 
Let $\varphi:[1, \infty) \to [1, \infty)$ be an increasing and strictly convex function satisfying $\varphi(1) = 1$ and $\varphi''(t) > c_0$ for some constant $c_0 > 0$ and for a.e. $t\in [1, \infty)$. Then, there exists an $\varepsilon_0 >0$ such that if 
\begin{equation}
\label{Inequality-Almost-Equality-thm-lin}
\int_{Q_1} \varphi(K(z,f))\,d\mathcal{L}^2(z)\leq (1+\varepsilon)\int_{Q_1} \varphi(K(z,f^*))\,d\mathcal{L}^2(z),
\end{equation}
holds for some $0<\varepsilon < \varepsilon_0$ and $f \in \mathcal{F}$,  
then the mapping $$
\Psi=f\circ (f^*)^{-1}: Q_2 \to \CC 
$$ 
satisfies the  estimate
\begin{equation}
\label{PsiConjugateDer-eq}
\int_{Q_2} |\Psi_{\bar{w}}(w)| \,d\mathcal{L}^2(w)\ll \sqrt{\varepsilon},
\end{equation}
where $Q_2= f^{\ast}(Q_1)$.
Moreover, the factor $\sqrt{\varepsilon}$ in this statement is sharp. 
\end{proposition} 

We can interpret this statement in the following way: If  inequality \eqref{Inequality} is an "almost equality" in the sense of \eqref{Inequality-Almost-Equality-thm-lin}, then 
$$
\Psi=f\circ (f^*)^{-1}
$$ 
is "almost conformal"; namely,  $|\Psi_{\bar{z}}|$ is "negligible". 

In turn, the proof of Proposition \ref{prop:lin-quant} relies on several lemmata, the first one being a standard Taylor-type formula for uniformly convex functions.

\begin{lemma}
\label{Taylor-convex-function-lemma}
If $c>0$ and $\varphi:I\to\R$ is a convex function on an open interval $I\subset\R$ satisfying $\varphi''(t)\geq c$ for a.e. $t\in I$, then for every $t,s\in I$, we have
\begin{equation}
\label{Taylor-convex-function}
\varphi(t)\geq \varphi(s)+\varphi'_+(s)(t-s)+\frac{c}2\,(t-s)^2,
\end{equation}
where $\varphi'_+$ denotes the  right-hand derivative of $\varphi$.
\end{lemma}

\begin{proof}
It is enough to prove that for all  $t,s\in I$ and for any subgradient $a\in\partial\varphi(s)$, we have
\begin{equation*}
\varphi(t)\geq \varphi(s)+a(t-s)+\frac{c}2\,(t-s)^2.
\end{equation*}

Assume first that $t>s$. By Theorem 1.3.1 in Kannan and Krueger \cite{KK}, the monotonicity of $\varphi'_+$ implies that $\varphi'_+$ is almost everywhere differentiable and satisfies
$$
\varphi'_+(t)-\varphi'_+(s)\geq \int_s^t\varphi''(\tau)\,d\tau.
$$
Hence, the lower bound $\varphi''(\tau)\geq c$ yields that $\varphi'_+(t)\geq \varphi'_+(s)+c(t-s)$. Since $\varphi$ is Lipschitz on $[s,t]$, it follows that 
$$
\varphi(t)-\varphi(s)=\int_s^t\varphi'_+(\tau)\,d\tau\geq   \int_s^t\varphi'_+(s)+c(\tau-s)\,d\tau=\varphi'_+(s)(t-s)+\frac{c}2\,(t-s)^2\geq a(t-s)+\frac{c}2\,(t-s)^2,
$$
which yields the desired inequality.

If $t<s$, we apply the previous argument to the function  $\psi(\tau)= \varphi(2s-\tau)$ satisfying that $\psi(s)=\varphi(s)$, $\psi''(\tau)\geq c$ a.e. and $-a\in\partial \psi(s)$.
\end{proof}

\begin{lemma}
\label{lemma:K(z,f)}
Let $\varphi:[1, \infty) \to [1, \infty)$ be an increasing and strictly convex function. 
If \eqref{Inequality-Almost-Equality-thm-lin} holds for some small $\varepsilon>0$ and $f \in \mathcal F$, then
\begin{equation}
\label{K(z,f)-eq}
\int_{Q_1}(K(z,f)-K(z,f^*))^2\,d\mathcal{L}^2(z)\ll\varepsilon.
\end{equation}
\end{lemma}
\begin{proof} First, recall that by Theorem 2  in \cite{FHS}, we have 
\begin{equation}\label{min_property_FengHuShen}
    \int_{Q_1} K(z,f)\,d\mathcal{L}^2(z)\geq \int_{Q_1} K(z,f^*)\,d\mathcal{L}^2(z).
\end{equation}
On the other hand, by the  Taylor formula \eqref{Taylor-convex-function} and the condition $\varphi'' >c>0$, we have for a.e. $z\in Q_1$ that 
$$
\varphi(K(z,f))-\varphi(K(z,f^*))\geq \varphi'(K(z,f^*))(K(z,f)-K(z,f^*))+\frac{c}2(K(z,f)-K(z,f^*))^2,
$$
where $\varphi'_+(K(z,f^*))\geq 0$. Integrating over $Q_1$ and applying \eqref{Inequality-Almost-Equality-thm-lin} yields that
\begin{align*}
\varepsilon\cdot\int_{Q_1} \varphi(K(z,f^*))\,d\mathcal{L}^2(z)\geq &
\int_{Q_1} \varphi(K(z,f))-\varphi(K(z,f^*))\,d\mathcal{L}^2(z) \\
\geq & ~ \frac{c}2 \int_{Q_1} (K(z,f)-K(z,f^*))^2\,d\mathcal{L}^2(z),
\end{align*}
which proves \eqref{K(z,f)-eq}.
\end{proof}

The following lemma provides an estimate similar to \eqref{Inequality-Almost-Equality-thm-lin} in the linear case  $\varphi(t) = t$.

\begin{lemma} 
\label{K-estimate}
Let $\varphi:[1, \infty) \to [1, \infty)$ be an increasing and strictly convex function. Then there exists a constant $C= C(\varphi) >0$ such that if \eqref{Inequality-Almost-Equality-thm-lin} holds for some small $\varepsilon>0$ and $f \in \mathcal F$, then
\begin{equation} \label{eq:Kestimate}
\int_{Q_1} K(z,f)\,d\mathcal{L}^2(z) \leq (1 + C\varepsilon)\int_{Q_1} K(z,f^{\ast})\,d\mathcal{L}^2(z).
    \end{equation}
\end{lemma}
\begin{proof} 
By Jensen’s inequality and the convexity of $\varphi$, we have
\begin{equation*}
  \varphi\left(\frac{1}{|Q_1|}\int_{Q_1} K(z,f) \, d\mathcal{L}^2(z)\right)  \leq   \frac{1}{|Q_1|}\int_{Q_1}\varphi(K(z,f)) \, d\mathcal{L}^2(z).
\end{equation*}
Using \eqref{Inequality-Almost-Equality-thm-lin}, it follows that 
\begin{equation*}
\frac{1}{|Q_1|}\int_{Q_1} \varphi(K(z,f))\,d\mathcal{L}^2(z)\leq
\frac{1}{|Q_1|}(1+\varepsilon)\int_{Q_1}\varphi(K(z,f^{\ast})) \, d\mathcal{L}^2(z) =   (1 +\varepsilon) \varphi(K(z,f^{\ast})).
\end{equation*}
Let $C>0$ be a constant to be determined later. By the convexity of $\varphi$ we can write  
\begin{equation*}
 \varphi\big(K(z,f^{\ast})(1+C\varepsilon)\big)   \geq \varphi(K(z, f^{\ast}) ) + 
\varphi'(K(z, f^{\ast}))\cdot C\varepsilon K(z, f^{\ast}).
\end{equation*} 
Choosing the value 
$$ C= C(\varphi) := \frac{\varphi(K(z, f^{\ast}))}{\varphi'(K(z, f^{\ast}) ) \cdot K(z, f^{\ast})}, $$
we obtain 
$$\varphi\big(K(z,f^{\ast})(1 + C\varepsilon)\big) \geq (1 +\varepsilon) \varphi(K(z,f^{\ast})).$$
Combining the above estimates, it follows that
$$\varphi\left(\frac{1}{|Q_1|}\int_{Q_1} K(z,f) \, d\mathcal{L}^2(z)\right) \leq \varphi\big(K(z,f^{\ast})(1 + C\varepsilon)\big) , $$
which implies by the injectivity of $\varphi$ that 
$$\frac{1}{|Q_1|}\int_{Q_1} K(z,f) \, d\mathcal{L}^2(z) \leq K(z,f^{\ast})(1 + C\varepsilon). $$
Multiplying this inequality by $|Q_1|$ completes the proof. 
\end{proof}

The following statement is technical, however, it plays an important role in the proof of Proposition~\ref{prop:lin-quant}.

\begin{lemma}
 \label{ImaginarySmall}
 Let $\varphi:[1, \infty) \to [1, \infty)$ be an increasing and strictly convex function. 
If \eqref{Inequality-Almost-Equality-thm-lin} holds for some small $\varepsilon>0$ and $f \in \mathcal F$, then
 there exists an angle $\alpha\in(-\pi,\pi]$ such that the following estimates hold:
\begin{align}
\label{RealPartSmall-eq}
\int_{Q_1}\left|e^{i\alpha}\mu^* f_z\right|-{\rm Re}\left(e^{i\alpha}\mu^* f_z\right)+\left|e^{i\alpha}f_{\bar{z}}\right|-{\rm Re}\left(e^{i\alpha}f_{\bar{z}}\right)\,d\mathcal{L}^2\ll & ~ \varepsilon\\
\label{ImaginarySmall-eq}
\int_{Q_1}\left|{\rm Im}\left(e^{i\alpha}\mu^* f_z\right)\right|+\left|{\rm Im}\left(e^{i\alpha}f_{\bar{z}}\right)\right|\,d\mathcal{L}^2\ll & ~ \sqrt{\varepsilon}.
\end{align}
\end{lemma}

\begin{proof}
We split the proof into two steps, corresponding to relations
\eqref{RealPartSmall-eq} and \eqref{ImaginarySmall-eq}.

To prove the first inequality, we start by recalling the following a chain of inequalities from the proof of Theorem 2 in \cite[relation (2.10)]{FHS},  which holds for any $f\in \mathcal{F}$:

\begin{align}
 k\ell \int_{Q_1}K(z, f^{\ast}) \,d\mathcal{L}^2(z) &\leq \left| \int_{Q_1} \left(\frac{\mu^{\ast}}{|\mu^{\ast}|}f_z + f_{\bar{z}}\right) \,d\mathcal{L}^2\right|^{2} \leq  \left( \int_{Q_1} \left|\frac{\mu^{\ast}}{|\mu^{\ast}|}f_z + f_{\bar{z}}\right| \,d\mathcal{L}^2\right)^{2} \leq \nonumber \\ & \leq \left(\int_{Q_1}\left|f_z\right|+\left|f_{\bar{z}}\right|\,d\mathcal{L}^2\right)^2 \leq k\ell \int_{Q_1}K(z, f) \,d\mathcal{L}^2(z) . \label{chain_ineq}
\end{align}
Now, Lemma \ref{K-estimate} yields the existence of a constant $C = C(\varphi)>0$ such that \eqref{eq:Kestimate} is true, which implies that the ratio of the left-hand side and the right-hand side of the above chain is greater or equal than $\frac{1}{1 + C\varepsilon}$. This, in turn yields that the ratio of the second and fourth term 
satisfies
$$ \frac{\left| \int_{Q_1} \left(\frac{\mu^{\ast}}{|\mu^{\ast}|}f_z + f_{\bar{z}}\right) \,d\mathcal{L}^2\right|}{\int_{Q_1}\left|f_z\right|+\left|f_{\bar{z}}\right|\,d\mathcal{L}^2} \geq (1+C\varepsilon)^{-\frac{1}{2}}.$$
Furthermore, if $0< C\varepsilon< 1$ we have and $(1+C\varepsilon)^{-\frac12}>1-C\varepsilon$, thus for sufficiently small $\varepsilon$ we obtain the estimate 
\begin{equation}
\label{fz-abs-difference}
\left|\int_{Q_1}\frac{\mu^*}{|\mu^*|} f_z+f_{\bar{z}}\,d\mathcal{L}^2\right|\geq (1-C\varepsilon)
\int_{Q_1}\left|f_z\right|+\left|f_{\bar{z}}\right|\,d\mathcal{L}^2.
\end{equation}
Moreover, the last inequality from \eqref{chain_ineq}  combined with \eqref{eq:Kestimate} from
Lemma~\ref{K-estimate} yields
\begin{equation}
\label{fz-abs-size}
\int_{Q_1}\left|f_z\right|+\left|f_{\bar{z}}\right|\,d\mathcal{L}^2\ll 1.
\end{equation}

Next, we choose $R>0$ and $\alpha\in(-\pi,\pi]$ such that
$$
\int_{Q_1} \left(\frac{\mu^*}{|\mu^*|} f_z+f_{\bar{z}} \right)\,d\mathcal{L}^2=Re^{-i\alpha}.
$$
Hence, \eqref{fz-abs-difference} yields that
\begin{align*}
\int_{Q_1}\left|f_z\right|+\left|f_{\bar{z}}\right|\,d\mathcal{L}^2\geq &
\int_{Q_1}{\rm Re}\left(e^{i\alpha}\frac{\mu^*}{|\mu^*|}  f_z\right)+{\rm Re}\left(e^{i\alpha}f_{\bar{z}}\right)\,d\mathcal{L}^2=R=\\
=&\left|\int_{Q_1}\frac{\mu^*}{|\mu^*|}  f_z+f_{\bar{z}}\,d\mathcal{L}^2\right|
\geq (1-C\varepsilon)
\int_{Q_1}\left|f_z\right|+\left|f_{\bar{z}}\right|\,d\mathcal{L}^2.
\end{align*}
This together with \eqref{fz-abs-size} implies the estimate 
$$\int_{Q_1}\left|f_z\right|+\left|f_{\bar{z}}\right|\,d\mathcal{L}^2- \left|\int_{Q_1}\frac{\mu^*}{|\mu^*|} \cdot f_z+f_{\bar{z}}\,d\mathcal{L}^2\right| \ll \varepsilon .$$
Since $|\mu^*|< 1$, we obtain
\begin{align*}
& \int_{Q_1}\left|e^{i\alpha}\mu^* f_z\right|-{\rm Re}\left(e^{i\alpha}\mu^* f_z\right)\,d\mathcal{L}^2 =  |\mu^{\ast}|\int_{Q_1} |f_z| + |f_{\bar{z}}| - |f_{\bar{z}}| -{\rm Re}\left(e^{i\alpha}\frac{\mu^*}{|\mu^*|} f_z\right)\,d\mathcal{L}^2\leq \\ &
\leq |\mu^{\ast}|\int_{Q_1} \left\{ |f_z| + |f_{\bar{z}}| \right\}- \left[{\rm Re}\left(e^{i\alpha}f_{\bar{z}}) \right)  +{\rm Re}\left(e^{i\alpha}\frac{\mu^*}{|\mu^*|} f_z\right)\right]\,d\mathcal{L}^2 \ll \varepsilon.
\end{align*}
In a similar way, we obtain 
$$\int_{Q_1}\left|e^{i\alpha}f_{\bar{z}}\right|-{\rm Re}\left(e^{i\alpha}f_{\bar{z}}\right)\,d\mathcal{L}^2
\ll \varepsilon,
$$
proving \eqref{RealPartSmall-eq}.

In order to prove \eqref{ImaginarySmall-eq}, we consider the continuous function $\Theta:\CC\to[0,\infty)$,
\[
\Theta(z)=
\begin{cases}
\frac{({\rm Im}\,z)^2}{2|z|}, & z\neq0,\\[1ex]
0, & z=0.
\end{cases}
\]
In particular, we have the estimate
\begin{equation} \label{Im-Re-abs}
\Theta(z)\leq |z|-{\rm Re}\,z, \quad \forall z \in \mathbb C.
\end{equation}
To show this, we use the inequality $1-\sqrt{1-t}\geq t/2$ for $t\in[0,1]$.  For all $z\neq 0$, we have
\begin{align*}
|z|-{\rm Re}\,z\geq&|z|-|{\rm Re}\,z|=|z|-\sqrt{|z|^2-({\rm Im}\,z)^2}\\
=&|z|\cdot\left(1-\sqrt{1-\frac{({\rm Im}\,z)^2}{|z|^2} }\right)\geq
\frac{({\rm Im}\,z)^2}{2|z|},
\end{align*}
proving \eqref{Im-Re-abs}.
For \eqref{ImaginarySmall-eq}, we use \eqref{Im-Re-abs}, the Cauchy-Schwarz inequality and \eqref{RealPartSmall-eq} to obtain
\begin{align*}
\int_{Q_1}\left|{\rm Im}\left(e^{i\alpha}\mu^* f_z\right)\right|\,d\mathcal{L}^2 & =
\int_{Q_1}\sqrt{2\Theta\left(e^{i\alpha}\mu^* f_z\right)\cdot \left|e^{i\alpha}\mu^* f_z\right|}\,d\mathcal{L}^2 \leq \\
&\leq \sqrt{2\int_{Q_1}\left|e^{i\alpha}\mu^* f_z\right|-{\rm Re}\left(e^{i\alpha}\mu^* f_z\right)\,d\mathcal{L}^2}\,\cdot \sqrt{\int_{Q_1}|f_z|\,d\mathcal{L}^2} \ll \sqrt{\varepsilon},
\end{align*}
and the estimate for $\int_{Q_1}\left|{\rm Im}\left(e^{i\alpha}f_{\bar{z}}\right)\right|\,d\mathcal{L}^2$ follows analogously.
\end{proof}

Before stating our next result, we make the following observation. Since for every $f \in \mathcal F$, the set $f(Q_1)$ is a fundamental domain for $L$, we have
\begin{equation}
\label{funddomain-fQ1}
\int_{Q_1}J(z,f)\,d\mathcal{L}^2(z) \leq \int_{f(Q_1)}1\,d\mathcal{L}^2 = \mathcal{L}^2(f(Q_1)) = k\ell,
\end{equation}
where the first inequality is the change of variables formula from Corollary~3.3.6 in Astala, Iwaniec, and Martin \cite[p. 57]{AIM09}. 

The last puzzle-piece needed to prove Proposition \ref{prop:lin-quant} is the following:

\begin{lemma}
\label{abs-difference}
Let $\varphi:[1,\infty)\to[1,\infty)$ be increasing and strictly convex.
If \eqref{Inequality-Almost-Equality-thm-lin} holds for some small $\varepsilon>0$ and $f \in \mathcal F$, then
$$
\int_{Q_1}\Big| \left|f_{\bar{z}}\right|- \left|\mu^*f_{z}\right| \Big|\,d\mathcal{L}^2\ll \sqrt{\varepsilon}.
$$
\end{lemma}
\begin{proof}
First, based on \eqref{funddomain-fQ1}, we have that
\begin{equation}
\label{intQ-square-diff}
\int_{Q_1}(|f_z|-|f_{\bar{z}}|)^2\,d\mathcal{L}^2\leq 
\int_{Q_1}|f_z|^2-|f_{\bar{z}}|^2\,d\mathcal{L}^2=\int_{Q_1}J(z,f)\,d\mathcal{L}^2\ll 1.
\end{equation}
Furthermore, by direct calculations, we obtain that for a.e. $z\in Q_1$
\begin{align*}
\left(K(z,f)-K(z,f^*)\right)^2=&\left(\frac{|f_z|+|f_{\bar{z}}|}{|f_z|-|f_{\bar{z}}|}-\frac{|f^*_z|+|f^*_{\bar{z}}|}{|f^*_z|-|f^*_{\bar{z}}|}\right)^2\\
=&\frac{(2|f^*_z|)^2}{(|f_z|-|f_{\bar{z}}|)^2(|f^*_z|-|f^*_{\bar{z}}|)^2}\cdot 
\left(\left|f_{\bar{z}}\right|- \left|\mu^*f_{z}\right|\right)^2.
\end{align*}
From the Cauchy-Schwarz inequality, \eqref{K(z,f)-eq} from Lemma  \ref{lemma:K(z,f)} and \eqref{intQ-square-diff}, it follows that
\begin{align*}
&\int_{Q_1}\Big| \left|f_{\bar{z}}\right|- \left|\mu^*f_{z}\right| \Big|\,d\mathcal{L}^2\leq \\&
\left(\int_{Q_1}\left(K(z,f)-K(z,f^*)\right)^2 \,d\mathcal{L}^2\right)^{\frac12}\cdot
\left(\int_{Q_1}\frac{(|f_z|-|f_{\bar{z}}|)^2(|f^*_z|-|f^*_{\bar{z}}|)^2}{(2|f^*_z|)^2}\,d\mathcal{L}^2\right)^{\frac12}
\ll \sqrt{\varepsilon}. 
\end{align*}
\end{proof}

After these preparations, we are ready to prove Proposition \ref{prop:lin-quant}:

\begin{proof}[Proof of Proposition \ref{prop:lin-quant}]
First, we recall the following version of the chain rule for complex derivatives (see, for instance, Lehto \cite{L}). 
Let $\Omega,\Omega'\subset\CC$ be open sets, $g\in W^{1,2}(\Omega')$, and let
$h\in C^1(\Omega)$ such that $h(\Omega)\subset \Omega'$. Then, for a.e. $z\in \Omega$, one has
\begin{align*}
\frac{\partial}{\partial z}(g\circ h)=&\left(\frac{\partial g}{\partial w}\circ h\right)\cdot h_z+ \left(\frac{\partial g}{\partial\bar{w}}\circ h\right)\cdot \overline{h_{\bar{z}}},
\\
\frac{\partial}{\partial\bar{z}}(g\circ h)=&\left(\frac{\partial g}{\partial w}\circ h\right)\cdot h_{\bar{z}}+ \left(\frac{\partial g}{\partial\bar{w}}\circ h\right)\cdot \overline{h_z}.
\end{align*}

Consequently, if $\Omega,\Omega'\subset\CC$ are open and $h:\Omega\to\Omega'$ is a $C^1$ diffeomorphism, by setting $g= h^{-1}$ we obtain that
\begin{align*}
\frac{\partial h^{-1}}{\partial w}=&\frac{\overline{h_z}}{|h_z|^2-|h_{\bar{z}}|^2},\\
\frac{\partial h^{-1}}{\partial \bar{w}}=&\frac{-h_{\bar{z}}}{|h_z|^2-|h_{\bar{z}}|^2},
\end{align*}
where, when the left-hand side is evaluated at a point $w\in \Omega'$, then the right-hand side is evaluated at $z=h^{-1}(w)\in\Omega$.
Applying the above identities to the map $$\Psi = f\circ(f^{\ast})^{-1}: Q_2 \to \mathbb C, ~ ~ \text{ where } ~ Q_2 = f^{\ast}(Q_1),$$ 
it follows that
\begin{align*}
 \Psi_{\bar{w}}=f_{\bar{z}}\cdot \frac{f^*_{z}}{|f^*_{z}|^2-|f^*_{\bar{z}}|^2}-f_z\cdot \frac{f^*_{\bar{z}}}{|f^*_{z}|^2-|f^*_{\bar{z}}|^2}=\frac{f^*_{z}}{|f^*_{z}|^2-|f^*_{\bar{z}}|^2}\cdot\left(f_{\bar{z}}-\mu^*f_{z}\right),
\end{align*}
where the right-hand side is evaluated at the point $z= (f^{\ast})^{-1}(w)$.

Let $\alpha\in(-\pi,\pi]$ be the constant given by Lemma~\ref{ImaginarySmall}. Using the Cauchy-Schwarz inequality and the linear change of variable $w= f^{\ast}(z)$ yields 
\begin{align}
\nonumber
\int_{Q_2} |\Psi_{\bar{w}}|(w) \,d\mathcal{L}^2(w)\leq 
&  ~ \frac{|f^*_{z}|}{|f^*_{z}|^2-|f^*_{\bar{z}}|^2}\cdot\int_{Q_2} \left|f_{\bar{z}}-\mu^*f_{z}\right|((f^{\ast})^{-1}(w)) \,d\mathcal{L}^2(w)  \\ 
=& ~ \frac{|f^*_{z}|}{|f^*_{z}|^2-|f^*_{\bar{z}}|^2}J(f^{\ast}) \cdot\int_{Q_1} \left|f_{\bar{z}}-\mu^*f_{z}\right|(z) \,d\mathcal{L}^2(z) \nonumber \\
\label{PsiConjugateDer0-Re}
\ll & \int_{Q_1} \left|{\rm Re}\left(e^{i\alpha}f_{\bar{z}}\right)-
{\rm Re}\left(e^{i\alpha}\mu^*f_{z}\right)\right| \,d\mathcal{L}^2+\\
\label{PsiConjugateDer0-Im}
&+ \int_{Q_1} \left|{\rm Im}\left(e^{i\alpha}f_{\bar{z}}\right)\right|+
\left|{\rm Im}\left(e^{i\alpha}\mu^*f_{z}\right)\right| \,d\mathcal{L}^2. 
\end{align}
By \eqref{ImaginarySmall-eq} from Lemma \ref{ImaginarySmall}, we know that
the integral in \eqref{PsiConjugateDer0-Im} is of order at most $\sqrt{\varepsilon}$. For the remaining term \eqref{PsiConjugateDer0-Re}, we first apply the triangle inequality, and then use \eqref{RealPartSmall-eq} from Lemma \ref{ImaginarySmall} together with the estimate \eqref{Lemma23-L1-norm-small} from Step 3, obtaining
$$
\int_{Q_1}\Big|\Big({\rm Re}\left(e^{i\alpha}\mu^* f_z\right)-\left|e^{i\alpha}\mu^* f_z\right|\Big)-\Big({\rm Re}\left(e^{i\alpha}f_{\bar{z}}\right)-\left|e^{i\alpha}f_{\bar{z}}\right|\Big)\Big|\,d\mathcal{L}^2
+ \int_{Q_1}\Big|  \left|\mu^*f_{z}\right| -\left|f_{\bar{z}}\right|\Big|\,d\mathcal{L}^2\ll \sqrt{\varepsilon}.
$$
This completes the proof.
\end{proof}

\begin{proof}[Proof of Theorem \ref{Thm:Lin-Stretch}]
Set $Q_2 = f^{\ast}(Q_1)$ and let $\Psi = f\circ (f^{\ast})^{-1} \in W^{1,1}(Q_2)$.
By the change of variables $\xi = f^{\ast}(z)$, we have
$$ \int_{Q_1}|f-f^{\ast}|(z) \, d\mathcal{L}^2(z) = \int_{Q_2} |\Psi (\xi)- \xi| J(\xi, (f^{\ast})^{-1}) \, d\mathcal{L}^{2}(\xi) \ll \int_{Q_2} |\Psi (\xi)- \xi|  \, d\mathcal{L}^{2}(\xi). $$
Applying the Cauchy-Pompeiu formula  
$$ \Psi(\xi) = \frac{1}{2\pi i}\int_{\partial Q_2} \frac{\Psi(w)}{w-\xi} dw + 
\frac{1}{\pi }\int_{Q_2}\frac{\Psi_{\bar{w}}(w)}{w-\xi}\, d\mathcal{L}^2(w)$$
and the Cauchy integral formula for the identity function 
$$ \xi= \frac{1}{2\pi i}\int_{\partial Q_2} \frac{w}{w-\xi} dw,$$
we obtain 
\begin{eqnarray*} \label{eq:C-P} \int_{Q_2} |\Psi (\xi)- \xi|  \, d\mathcal{L}^{2}(\xi) &\ll \ds\int_{Q_2}
\int_{\partial Q_2} \frac{|\Psi(w)-w|}{|w-\xi|} \, dw \, d\mathcal{L}^2(\xi)  + \int_{Q_2}\int_{Q_2} 
\frac{|\Psi_{\bar{w}}(w)|}{|w-\xi|}\, d\mathcal{L}^2(w)  \, d\mathcal{L}^2(\xi) \nonumber \\
 &= \ds\int_{\partial Q_2}\int_{Q_2}
\frac{|\Psi(w)-w|}{|w-\xi|} \, d\mathcal{L}^2(\xi) dw + \int_{Q_2}\int_{Q_2} 
\frac{|\Psi_{\bar{w}}(w)|}{|w-\xi|}\, d\mathcal{L}^2(\xi) \, d\mathcal{L}^2(w)   ,
\end{eqnarray*}
where we used Fubini's theorem in the last equality.
Indeed, integration in polar coordinates shows that there exists a constant $C>0$ such that for all $w \in Q_2 $,
$$ \int_{Q_2} 
\frac{1}{|w-\xi|}\, d\mathcal{L}^2(\xi) \leq C,$$
where $C$ depends only on the domain $Q_2$, i.e., on the constants $k$,$\ell$, and $n$. Therefore, it follows that
\begin{equation} \label{eq:C-P} \int_{Q_2} |\Psi (\xi)- \xi|  \, d\mathcal{L}^{2}(\xi) \ll \ds\int_{\partial Q_2}
|\Psi(w)-w| \, dw + \int_{Q_2} 
|\Psi_{\bar{w}}(w)|\, d\mathcal{L}^2(w).
\end{equation}
Combining this estimate with the boundary assumption \eqref{eq:boundary-cond}, we obtain that the first integral on the right-hand side of \eqref{eq:C-P} is $\ll \sqrt{\varepsilon}$. For the estimate of the second integral we  use Proposition \ref{prop:lin-quant}. 
\end{proof}

In what follows we demonstrate the sharpness of the factor $\sqrt{\varepsilon}$ both in the statement of Proposition \ref{prop:lin-quant} and Theorem \ref{Thm:Lin-Stretch}. The construction is based on Example \ref{ex:non-unique}. 

\begin{example} \label{ex:lin-sharp}
Let $\ell=1$, $k>1$, $Q_1=[0,1]\times[0,1]$, $Q_2 = [0,k]\times [0,1]$, and consider the linear stretch-map $f^{\ast}: Q_1 \to Q_2$ given by $f^{\ast}(x+iy) = kx +iy$. For any $\varepsilon < (k-1)^2$ there exists a map $f_{\varepsilon} \in \mathcal{F}$ such that 
\begin{equation*}
 \int_{Q_1} K^2(z, f_{\varepsilon}) d\mathcal{L}^2 =\int_{Q_1} K^2(z, f^{\ast}) d\mathcal{L}^2 + \varepsilon \ \text{ and } \ \int_{Q_2} |\Psi^{\varepsilon}_{\bar{w}}| d\mathcal{L}^2 = \frac{\sqrt{\varepsilon}}{4},    
\end{equation*}
where $\Psi^{\varepsilon}= f_{\varepsilon}\circ (f^{\ast})^{-1}$. 
 Furthermore, for any $\xi \in \partial Q_1$ we have that $|f_{\varepsilon}(\xi) - f^{\ast}(\xi)|\ll \sqrt{\varepsilon} $. On the other hand,
$$\int_{Q_1}|f_\varepsilon-f^{\ast}|(z) \, d\mathcal{L}^2(z) \gg \sqrt{\varepsilon}.$$
\end{example}

\begin{proof}
First, let us recall the notation  $K(z, f^{\ast}) = k$, hence
$$ 
\int_{Q_1} K(z, f^{\ast})^2 d\mathcal{L}^2= k^2. 
$$

The map $f_{\varepsilon}: Q_1 \to Q_2 $ is defined exactly as in Example \ref{ex:non-unique}. 
Recall that for small $\varepsilon >0$ and $z= x +iy$, where $x\neq \frac{1}{2}$,
$$
K(z, f_{\varepsilon}) =
\left\{\begin{array}{rl}
k+ \sqrt{\varepsilon} & \mbox{ if } \ x\in (0,\frac{1}{2}), \\ k-\sqrt{\varepsilon} 
&\mbox{ if } \ x \in (\frac{1}{2}, 1), 
\end{array}\right.
$$
 which gives that 
 $$ \int_{Q_1} K(z, f_{\varepsilon})^2 d\mathcal{L}^2 =\int_{Q_1} K(z, f^{\ast})^2 d\mathcal{L}^2 + \varepsilon.$$

Next, observe that $(f^{\ast})^{-1}: Q_2 \to Q_1$ is given by
$(f^{\ast})^{-1}(w) = \frac{x'}{k}+iy'$ for $w= x'+iy' \in Q_2$. Therefore, 
we obtain for $\Psi^{\varepsilon}: Q_2 \to Q_2$ the explicit formula 
 $$
\Psi^{\varepsilon}(w) = f_{\varepsilon}\circ (f^{\ast})^{-1}(w)=
\left\{\begin{array}{rl}
\frac{k+ \sqrt{\varepsilon}}{k}x' +iy'& \mbox{ if } \ x'\in (0,\frac{k}{2}), \\ \frac{k-\sqrt{\varepsilon}}{k}x' + \sqrt{\varepsilon}+iy'
&\mbox{ if } \ x' \in (\frac{k}{2}, k), 
\end{array}\right.
$$
By direct calculation, we obtain
$$
\Psi^{\varepsilon}_{\bar{w}} = \frac{1}{2}\left[\frac{\partial \Psi^{\varepsilon}}{\partial x'} +i\frac{\partial \Psi^{\varepsilon}}{\partial y'}\right]=
\left\{\begin{array}{rl}
\frac{1}{2}[\frac{k+ \sqrt{\varepsilon}}{k} -1]& \mbox{ if } \ x'\in (0,\frac{k}{2}), \\ \frac{1}{2} [\frac{k-\sqrt{\varepsilon}}{k} -1]
&\mbox{ if } \ x' \in (\frac{k}{2}, k). 
\end{array}\right.
$$
It follows that 
$$ \int_{Q_2} |\Psi^{\varepsilon}_{\bar{w}}| d\mathcal{L}^2 = \frac{\sqrt{\varepsilon}}{4},$$
which proves the sharpness of the factor $\sqrt{\varepsilon}$ in Proposition \ref{prop:lin-quant}. 
Finally, a straightforward computation shows that
\[
\int_{Q_1}|f_\varepsilon-f^{\ast}|\,d\mathcal L^2 = \sqrt{\varepsilon} \int_0^\frac{1}{2} x dx + \sqrt{\varepsilon} \int_\frac{1}{2}^1 (1-x) dx = \frac{\sqrt{\varepsilon}}{4},
\]
showing the sharpness of the rate $\sqrt{\varepsilon}$  in
Theorem~\ref{Thm:Lin-Stretch}.
\end{proof}

Next, we show that the uniform convexity assumption on $\varphi$ in  Proposition \ref{prop:lin-quant} and Theorem \ref{Thm:Lin-Stretch} -- namely, that there exists a constant $c_0>0$ such that
$\varphi''(t) > c_0$ for a.e. $t\in [1, \infty)$ --  plays an indispensable role. Again, we use elements from Example \ref{ex:non-unique} and  \ref{ex:lin-sharp}.

\begin{example} \label{ex:unif-convexity-assumption}
Let $\ell=1$, $k>1$, $Q_1=[0,1]\times[0,1]$, $Q_2 = [0,k]\times [0,1]$, and  $f^{\ast}: Q_1 \to Q_2$, $f^{\ast}(x+iy) = kx +iy$ be the linear stretch. Consider the family of functions 
$$\varphi_c: [1, \infty) \to \mathbb R, \varphi_c(t) = t+c t^2,$$
where $c >0$. 
Then,  for any $0<\varepsilon < (k-1)^2$ there exists a map $f_{\varepsilon} \in \mathcal{F}$ such that 
\begin{equation*}\label{Example2.3-deficit}
\int_{Q_1} \varphi_c(K(z,f_\varepsilon))\,d\mathcal{L}^2(z)\leq \left(1+\varepsilon\right)\int_{Q_1} \varphi_c(K(z,f^*))\,d\mathcal{L}^2(z).
\end{equation*}
However, for the map $\Psi^{\varepsilon}= f_{\varepsilon}\circ (f^{\ast})^{-1}: Q_2 \to Q_2 $, we have that
\begin{equation*}
     \int_{Q_2} |\Psi^{\varepsilon}_{\bar{w}}| d\mathcal{L}^2  \gg \sqrt{\frac{\varepsilon}{c}}.
 \end{equation*}
Furthermore, we have $|f_{\varepsilon}(\xi) - f^{\ast}(\xi)|\ll \sqrt{\varepsilon} $ for any $\xi \in \partial Q_1$. Nevertheless,
$$\int_{Q_1}|f_\varepsilon-f^{\ast}|(z) \, d\mathcal{L}^2(z) \gg \sqrt{\frac{\varepsilon}{c}}.$$
Letting $c \to 0$ shows the sharpness of the uniform convexity condition present in the assumptions of Proposition \ref{prop:lin-quant} and Theorem \ref{Thm:Lin-Stretch}. 
\end{example}
\begin{proof}
First, using the notation  $K(z, f^{\ast}) = k$, we have
$$ 
\int_{Q_1} \varphi_c(K(z, f^{\ast})) d\mathcal{L}^2= k+ck^2. 
$$
For every $0<\varepsilon < (k-1)^2$ we define the map $f_{\varepsilon}: Q_1 \to Q_2 $, $f_{\varepsilon}(x+iy) = g_{\varepsilon}(x) +iy$, where $g_{\varepsilon}:[0,1] \to [0,k]$
is given by
$$
g_{\varepsilon}(x)=
\left\{\begin{array}{rl}
\left(k+ \sqrt{\left(k^2+\frac{k}{c}\right)\varepsilon}\right)x, & \mbox{ if } \ x\in [0,\frac{1}{2}], \\ \left(k- \sqrt{\left(k^2+\frac{k}{c}\right)\varepsilon}\right)x + \sqrt{\left(k^2+\frac{k}{c}\right)\varepsilon}, 
& \mbox{ if } \ x \in [\frac{1}{2}, 1].
\end{array}\right.
$$
Therefore, for every $z \in Q_1, z= x +iy$ with $x\neq \frac{1}{2}$, we have
$$
K(z, f_{\varepsilon}) =
\left\{\begin{array}{rl}
k+ \sqrt{k^2+\frac{k}{c}} \sqrt{\varepsilon} & \mbox{ if } \ x\in (0,\frac{1}{2}), \\ k- \sqrt{k^2+\frac{k}{c}}\sqrt{\varepsilon} 
&\mbox{ if } \ x \in (\frac{1}{2}, 1).
\end{array}\right.
$$
Hence, a direct calculation yields that
$$
\int_{Q_1} \varphi_c( K(z, f_{\varepsilon}) ) d\mathcal{L}^2 = (k + ck^2)(1+\varepsilon) = (1+\varepsilon)\int_{Q_1}  \varphi_c(K(z, f^{\ast})) d\mathcal{L}^2.
$$
On the other hand, using the calculations from Example \ref{ex:lin-sharp}, we obtain that
\[
\int_{Q_1}|f_\varepsilon-f^{\ast}|\,d\mathcal L^2 = \int_{Q_2} |\Psi^\varepsilon_{\bar w}|\,d\mathcal L^2
= \frac{1}{4}\sqrt{\frac{ck^2+k}{c}}\sqrt{\varepsilon} \gg \sqrt{\frac{\varepsilon}{c}}.
\]
\end{proof}

\section{Quantitative stability for the spiral stretch map} 
\label{secS3}

In this section we prove Theorem~\ref{Thm:Spiral-Stretch-stab} by using Proposition~\ref{prop:lin-quant} from the previous section. The overall strategy follows the proof of Theorem~\ref{Thm:Lin-Stretch}, the main difference being the use of exponential and logarithmic coordinates.  

In the following, let $q\in(0,1)$, $k>0$ and $\theta \in [-\pi, \pi]$ be fixed, and consider the two annuli
\begin{equation*} \label{annuli-weak}
A_1 = \{ w\in \CC: q \leq |w| \leq 1 \}, \ \text{and} \ A_2 = \{ w\in \CC: q^k \leq |w| \leq 1 \},
\end{equation*}
for which the corresponding spiral-stretch map $g^{\ast}: A_1 \to A_2$ is given by 
\begin{equation*} \label{eq:spiral-stretch-weak}
g^{\ast}(w)= w|w|^{k-1}\exp\left(\frac{\theta\log|w|}{\log q}\cdot i\right).
\end{equation*}

\begin{proof}[Proof of Theorem \ref{Thm:Spiral-Stretch-stab}]
Let $g:A_1\to A_2$ be an orientation-preserving homeomorphism  with finite distortion in $W^{1,1}(A_1, A_2)$ such that
$g=g^{\ast}$ on $\partial A_1$.
For $N\in \N$ we consider the $N$th spiral-stretch map 
$$ g_N(w)=w\cdot |w|^{k-1} \exp\left(i\cdot \frac{\theta+2N\pi}{\log q}\cdot \log|w|\right), \quad \ w\in A_1,$$
and note that all these maps satisfy the boundary condition
$\ds g_N|_{\partial A_1}=g^{\ast}|_{\partial A_1}$.

It is well-known (see \cite{BFP}, or \cite{FHS})  that there exists an $N\in \N$ such that $g$ and $g_N$ are homotopic with respect to $\partial A_1$. By Theorem 5 in \cite{BFP} it follows that $g_N$ minimizes the mean distortion in its own homotopy class, i.e.,
$$\int_{A_1} \frac{\varphi(K(w,g))}{|w|^2}\,d\mathcal{L}^2(w) \geq \int_{A_1} \frac{\varphi(K(w,g_N))}{|w|^2}\,d\mathcal{L}^2(w).  $$
Furthermore, it follows from the proof of Theorem 6 in \cite{BFP} that there exists a constant $C = C(q,k,\theta,\varphi) >0$ such that for every $N\geq 1$, 
\begin{equation*}
\label{N-atleast-1}
\int_{A_1} \frac{\varphi(K(w,g_N))}{|w|^2}\,d\mathcal{L}^2(w)>\int_{A_1} \frac{\varphi(K(w,g^{\ast}))}{|w|^2}\,d\mathcal{L}^2(w)+C.
\end{equation*}
Consequently, if $g: A_1 \to A_2$ satisfies the conditions of Theorem \ref{Thm:Spiral-Stretch-stab}, namely, that 
\begin{equation}
\label{Feng-Hu-Shen-epsilon}
\int_{A_1} \frac{\varphi(K(w,g))}{|w|^2}\,d\mathcal{L}^2(w)\leq(1+\varepsilon) 
\int_{A_1} \frac{\varphi(K(w,g^{\ast}))}{|w|^2}\,d\mathcal{L}^2(w)
\end{equation}
holds for small $\varepsilon>0$, then $N=0$ and $g$ must be in fact homotopic to $g^{\ast}$. 

Let us define the numbers 
\begin{equation*}
\ell= \frac1{2\pi}\log \frac1q,\  \ \text{and} \ 
n=-\frac{\theta+2N\pi}{2\pi \ell},
\end{equation*} 
and consider the associated rectangle $Q_1= [0,\ell]\times [0,1]$, and linear stretch map  
$$f^{\ast}(x+iy) = kx+ inx +iy$$
as in Section \ref{secS2}. 

Let $\gamma$ be the interval $[q,1] \subseteq A_1$, then $g(\gamma)$, resp. $g^{\ast}(\gamma)$ are two homotopic simple arcs in $A_2$. Then 
\begin{equation*} 
 z\mapsto w= q\exp(2\pi z)
\end{equation*} 
is a conformal mapping from the interior of $Q_1$ onto $A_1\setminus (\partial A_1 \cup \gamma) $. 
In the other direction we consider the map
\begin{equation} \label{eq:G} 
G: w \mapsto \frac{1}{2\pi}\log w +k\ell +in\ell,
\end{equation}
that is conformal from the domain $A_2\setminus (\partial A_2\cup g(\gamma))$ onto a Jordan domain $Q_g$ that is bounded by the segment $\sigma_1$ connecting $0$ to $i$, the segment $\sigma_1+k\ell +in\ell$, the simple smooth curve $\sigma_2$ connecting $0$ to $k\ell +in\ell$, and the simple smooth curve $\sigma_2+i$. It follows that
\begin{equation}
\label{sigma2-within-strip}
{\rm Re}\,z\in(0,k\ell)\mbox{ \ for }z\in\sigma_2\backslash\{0,k\ell +in\ell\}.
\end{equation}

We note that by our choice of the parameters $\ell$ and $n$, the map from \eqref{eq:G} will be conformal from the domain $A_2 \setminus (\partial A_2 \cup g^{\ast}(\gamma)))$ onto a parallelogram $Q_2$ with vertices $\{0, k\ell, k\ell+i\ell, i\}$. 

Here, the branch of logarithm $\log w = \log|w| + i\arg w$ in \eqref{eq:G} is chosen in a way such that $\log 1 =0$ and $\arg w $ depends continuously on $w\in A_2\setminus (\partial A_2\cup g(\gamma))$, resp. $w \in A_2 \setminus (\partial A_2 \cup g^{\ast}(\gamma)))$.  In particular $\arg w$ extends continuously to each side of $g(\gamma)$ (resp. $g^{\ast}(\gamma)$) but has a jump $2\pi i$ across $g(\gamma)$ (resp. $g^{\ast}(\gamma)$). 
This implies that 
\begin{equation}
    f^{\ast}(z)= \frac{1}{2\pi}\log g^{\ast}(q \exp(2\pi z)) +k\ell+in\ell = kx +inx +iy
\end{equation}
is in fact the linear stretch map $f^{\ast}: Q_1 \to Q_2$, while the map 
\begin{align}
\label{f-def-from-g}
    f(z)= &\frac{1}{2\pi}\log g\left(q \exp(2\pi z)\right) +k\ell+in\ell &&\mbox{if }z\in Q_1,\\
    \label{f-def-from-Gg}
    =&G\circ g\left(q \exp(2\pi z)\right)&&\mbox{if }z\in {\rm int}\,Q_1
\end{align}
is an element of the class $\mathcal{F}$ satisfying
$f(Q_1)=Q_g$ by \eqref{f-def-from-Gg}.
The properties \eqref{condf0}, \eqref{condfx} and \eqref{condfy} readily hold for $f$. In particular, taking other branches of logarithm instead of the one in the definition of \eqref{f-def-from-Gg} yields that the translates of $f(Q_1)=Q_g$ by vectors of the form $mi$, where $m\in\Z$, tile the parallel strip bounded by the "vertical" lines $\R i$ and $k\ell+\R i$ (cf. \eqref{sigma2-within-strip}). Namely, the union of the translates is the strip, and the interiors of the translates are pairwise disjoint. In turn, we conclude that the translates of $f(Q_1)$ by vectors in the lattice $L=\Z i+\Z(k\ell+in\ell)$ tile $\CC$.

By a change of variables and the invariance of distortion under composition by conformal maps, we obtain that
$$ \int_{A_1} \frac{\varphi(K(w, g))}{|w|^2} d\mathcal{L}^2(w) = 
4 \pi^2 \int_{Q_1}\varphi(K(z,f)) d\mathcal{L}^2(z),$$
and  
$$ \int_{A_1} \frac{\varphi(K(w, g^*))}{|w|^2} d\mathcal{L}^2(w) = 
4 \pi^2 \int_{Q_1}\varphi(K(z,f^*)) d\mathcal{L}^2(z).$$
Using \eqref{Feng-Hu-Shen-epsilon} and the above relations yields
\begin{equation*}
\int_{Q_1} \varphi(K(z,f))\,d\mathcal{L}^2(z)\leq(1+\varepsilon) 
\int_{Q_1} \varphi(K(z,f^{\ast}))\,d\mathcal{L}^2(z).
\end{equation*}
Therefore, by Proposition \ref{prop:lin-quant}, we obtain that for $\Psi= f\circ (f^{\ast})^{-1}$ we have the estimate 
$$\int_{Q_2} |\Psi_{\bar{z}}(z)|\,d\mathcal{L}^2(z)\ll \sqrt{\varepsilon}.$$

Let us consider the map
$$\Phi: A_2 \to A_2, \ \Phi := g\circ (g^{\ast})^{-1}. $$
We observe that $\Phi(w)$ is well defined for all $w\in \partial A_2$ by the boundary conditions.
Furthermore, let us note that $\Phi = F\circ \Psi \circ G$, where both $F$ and $G$ are conformal maps,  $F: Q_g \to A_2$ is given by 
\begin{equation*} \label{eq:F} 
 F: z\mapsto w= q^k\exp(2\pi z),
\end{equation*} 
and $G: Q_2 \to A_2$ is given by \eqref{eq:G}.

By the chain rule we have the equality
$$ \Phi_{\bar{w}}(w) = F_{z}(\Psi(G(w)))\cdot \Psi_{\bar{z}}(G(w)) \cdot \overline{G_w(w)}.$$
Using the fact that $|F_z|\approx 1$, $|G_w| \approx 1$ and $|G^{-1}_{z}|= J(G^{-1}) \approx 1$, by the change of variable $z= G(w)$, we obtain 
\begin{align} \label{Phi-d-bar}
\int_{A_2} |\Phi_{\bar{w}}|(w) \,d\mathcal{L}^2(w) \approx&
\int_{A_2} |\Psi_{\bar{z}}(G(w))|\cdot \left|\overline{G_w(w)}\right|\ \,d\mathcal{L}^2(w)\nonumber\\
\approx & \int_{Q_2} |\Psi_{\bar{z}}(z)| \,d\mathcal{L}^2(z) \ll \sqrt{\varepsilon}.
\end{align}



Next, we show that \eqref{Phi-d-bar} implies that
\begin{equation} \label{ErrorforPhi-eq}
\int_{A_2} |\Phi(w)-w| \,d\mathcal{L}^2\ll \sqrt{\varepsilon}.
\end{equation}
In order to see this we shall apply the Cauchy-Pompeiu formula (see, e.g., \cite{L})
\begin{equation} \label{eq:E}
\Phi(w)=\frac1{2\pi i}\int_{\partial A_2} \frac{\Phi(\xi)}{\xi-w} \,d\xi+E(w),
\end{equation}
where
$$
E(w)=\frac1{\pi}\int_{A_2} \frac{\Phi_{\bar{w}}(\xi)}{\xi-w} \,d\mathcal{L}^2(\xi).
$$
Let us note that for fixed $\xi\in A_2$, 
$$
\int_{A_2}\frac1{|w-\xi|}d\mathcal{L}^2(w)\ll \int_{A_2-A_2}\frac1{|w|}d\mathcal{L}^2(w)\ll 1.
$$
Now, it follows from \eqref{Phi-d-bar}
 that
\begin{align} \label{E-est}
\int_{A_2}|E(w)|\,d\mathcal{L}^2(w)\ll &\int_{A_2} \int_{A_2}\frac{|\Phi_{\bar{w}}(\xi)|}{|w-\xi|}d\mathcal{L}^2(w) \,d\mathcal{L}^2(\xi)\nonumber\\
\ll &\int_{A_2}|\Phi_{\bar{w}}(\xi)|\,d\mathcal{L}^2(\xi)\ll \sqrt{\varepsilon}. 
\end{align}
Let us recall that  $\Phi(\xi)=\xi$ for $\xi\in \partial A_2$, therefore

 $$\frac1{2\pi i}\int_{\partial A_2} \frac{\Phi(\xi)}{\xi-w} \,d\xi=w  \ \text{for} \  w  \in {\rm int}\,A_2 .$$
Using the above relation together with \eqref{eq:E} and \eqref{E-est}, we obtain \eqref{ErrorforPhi-eq}.


Having the estimate \eqref{ErrorforPhi-eq} at hand we can finish the proof of Theorem \ref{Thm:Spiral-Stretch-stab} as follows. 

Observe that $J(w,(g^{\ast})^{-1})\approx 1$. This implies by the change of variable 
$z= (g^{\ast})^{-1}(w)$ that
$$
\int_{A_1} |g-g^{\ast}|(z) \,d\mathcal{L}^2(z)=
\int_{A_2} |\Phi(w)-w| J(w,(g^{\ast})^{-1}) \,d\mathcal{L}^2(w)\ll \sqrt{\varepsilon}.
$$
\end{proof}

\begin{remark} 
Observe that similarly to Theorem \ref{Thm:Lin-Stretch}, the condition $g= g^{\ast}$ on 
$\partial A_1$ can be relaxed to the assumption 
$$ \int_{\partial A_1}|g-g^{\ast}| \ll \sqrt{\varepsilon},$$
to obtain the same conclusion. 
\end{remark}

In the following, we prove the sharpness of 
the factor $\sqrt{\varepsilon}$ in  Theorem \ref{Thm:Spiral-Stretch-stab}. The arguments are based on Example \ref{ex:lin-sharp}.

\begin{example} \label{ex:stretch-sharp}
Let  $0<q<1$ and $k>1$, and for the annuli 
\begin{equation*}
A_1=\{w\in\CC:q\leq |w|\leq 1\} \ \text{and} \ A_2=\{w\in\CC:q^k\leq |w|\leq 1\}, 
\end{equation*}
consider the radial stretch map
$$ 
g^{\ast} : A_1 \to A_2, \ g^{\ast}(w) = w\cdot |w|^{k-1}.
$$
For small $\varepsilon >0$, we construct a quasiconformal mapping $g=g^{(\varepsilon)}:A_1 \to A_2$ such that $g|_{\partial A_1} = g^{\ast}|_{\partial A_1}$ with the properties that 
\begin{description}

\item[(i)] We have $g^{(\varepsilon)}\neq g^{(\varepsilon')}$ for $\varepsilon\neq \varepsilon'\in(0,(k-1)^2)$, and 
\begin{equation} \label{eq:dist-eps-equality}
\int_{A_1} \frac{K(w,g)}{|w|^2}\,d\mathcal{L}^2(w)= 
\int_{A_1} \frac{K(w,g^{\ast})}{|w|^2}\,d\mathcal{L}^2(w).
\end{equation}

\item[(ii)] If $\varphi(t)=t^2$, then
\begin{equation} \label{eq:dist-eps}
\int_{A_1} \frac{\varphi(K(w,g))}{|w|^2}\,d\mathcal{L}^2(w)\leq(1+\varepsilon) 
\int_{A_1} \frac{\varphi(K(w,g^{\ast}))}{|w|^2}\,d\mathcal{L}^2(w),
\end{equation}
and  
\begin{equation}
 \label{eq:sqrt-eps}
    \int_{A_1} |g-g^{\ast}|(w) d \mathcal{L}^2(w) \gg \sqrt{\varepsilon}.
\end{equation}
\end{description}
\end{example}

\begin{proof}
Let us note first that $g^{\ast}(w)= w^{\frac{k+1}{2}}\cdot \bar{w}^{\frac{k-1}{2}}$, which implies that 
$$
|g^{\ast}_w(w)| = \left(\frac{k+1}{2}\right)\cdot |w|^{k-1}\ \text{and}  \ |g^{\ast}_{\bar{w}}(w)| = \left(\frac{k-1}{2}\right)\cdot |w|^{k-1}, 
$$
hence
$$ K(w, g^{\ast})= \frac{|g^{\ast}_w(w)|+ |g^{\ast}_{\bar{w}}(w)|}{|g^{\ast}_w(w)|- |g^{\ast}_{\bar{w}}(w)|}= k.$$
Therefore, by using integration in polar coordinates, a straightforward computation gives that the mean distortion of the minimizer $g^{\ast}$ satisfies
\begin{equation} \label{eq:K-g-star}
\int_{A_1} \frac{\varphi(K(w,g^{\ast}))}{|w|^2}\,d\mathcal{L}^2(w) = 2\pi \left(\log\frac{1}{q}\right)k^2.
\end{equation}
For any given positive $\varepsilon < (k-1)^2$, let the map $g_\varepsilon = g:A_1 \to A_2$ be defined by
$$
g(w)=
\left\{\begin{array}{rl}
q^{\sqrt{\varepsilon}} \cdot w\cdot |w|^{(k-1 -\sqrt{\varepsilon})} & \mbox{ if } \ |w| \in [q,q^{\frac{1}{2}}], \\ w\cdot |w|^{(k-1+\sqrt{\varepsilon})} 
&\mbox{ if } \ |w| \in [q^{\frac{1}{2}}, 1].
\end{array}\right.
$$
Then, the complex derivatives are given by
$$
g_w(w)=
\left\{\begin{array}{rl}
q^{\sqrt{\varepsilon}} \cdot \frac{k+1-\sqrt{\varepsilon}}{2}\cdot |w|^{(k-1 -\sqrt{\varepsilon})} & \mbox{ if } \ |w| \in [q,q^{\frac{1}{2}}], \\ \frac{k+1+\sqrt{\varepsilon}}{2}\cdot |w|^{(k-1+\sqrt{\varepsilon})} 
&\mbox{ if } \ |w| \in [q^{\frac{1}{2}}, 1],
\end{array}\right.
$$
and  
$$
g_{\bar{w}}(w)=
\left\{\begin{array}{rl}
q^{\sqrt{\varepsilon}} \cdot \frac{k-1-\sqrt{\varepsilon}}{2}\cdot |w|^{(k-1 -\sqrt{\varepsilon})} & \mbox{ if } \ |w| \in [q,q^{\frac{1}{2}}], \\ \frac{k-1+\sqrt{\varepsilon}}{2}\cdot |w|^{(k-1+\sqrt{\varepsilon})} 
&\mbox{ if } \  |w| \in [q^{\frac{1}{2}}, 1].
\end{array}\right.
$$
Therefore, we obtain for the   distortion
$$
K(w, g)=
\left\{\begin{array}{rl}
k-\sqrt{\varepsilon} & \mbox{ \ if } \ |w| \in [q,q^{\frac{1}{2}}), \\ 
k+\sqrt{\varepsilon}
&\mbox{ \ if} \ \ |w| \in (q^{\frac{1}{2}}, 1].
\end{array}\right.
$$
Using this formula and integration in polar coordinates $w=r\exp(i\theta)$, we obtain \eqref{eq:dist-eps-equality} in (i) by \eqref{eq:K-g-star}, and the formula
\begin{equation} \label{eq:dist-g}
\int_{A_1} \frac{\varphi(K(w,g))}{|w|^2}\,d\mathcal{L}^2(w) = 2\pi \left(\log\frac{1}{q}\right)\cdot (k^2 +\varepsilon)
\end{equation}
for the mean distortion in \eqref{eq:dist-g}. Combining relations \eqref{eq:K-g-star} and \eqref{eq:dist-g} yields \eqref{eq:dist-eps} in (ii).

It remains to verify relation \eqref{eq:sqrt-eps}. To do that, we shall write $w$ in polar coordinates $w=r\exp(i\theta)$. In these coordinates the  map $g:A_1 \to A_2 $ will be given by the formula
$g(r\exp(i\theta))= \varrho_\varepsilon(r)\exp(i\theta)$, where $\varrho_\varepsilon: [q,1] \to [q^k, 1]$ satisfies 
$$
\varrho_\varepsilon(r)=
\left\{\begin{array}{rl}
r^{(k- \sqrt{\varepsilon})}\cdot q^{\sqrt{\varepsilon}} & \mbox{ \ if } \ r\in [q,q^{\frac{1}{2}}], \\ 
r^{(k+\sqrt{\varepsilon})} 
&\mbox{ \ if } \ r \in [q^{\frac{1}{2}}, 1].
\end{array}\right.
$$
It follows that 
$$  \int_{A_1} |g-g^{\ast}|(w) d \mathcal{L}^2(w)= 2 \pi \int_q^1 |r^k-\varrho_\varepsilon (r)|r dr = 2 \pi \int_q^1 (r^k-\varrho_\varepsilon(r))r dr ,  $$
since $r^k >\varrho_\varepsilon(r)$ for any $r\in (q, 1)$. 
Using a Taylor expansion in terms of $\sqrt{\varepsilon}$ of the function $\varrho_\varepsilon$ for fixed $r \in [q^{\frac{1}{4}} , q^{\frac{3}{4}}]$, we deduce the existence of $C>0$ depending on $q$ such that
$$ 
(r^k-\varrho_\varepsilon(r))r \geq C \sqrt{\varepsilon}, \mbox{ for } r \in [q^{\frac{1}{4}} , q^{\frac{3}{4}}]. 
$$
In turn, we  estimate the integral as 
$$ 
\int_q^1 (r^k-\varrho_\varepsilon(r))r dr \gg \sqrt{\varepsilon},
$$
proving \eqref{eq:sqrt-eps}. This finishes the proof of the statements in Example~\ref{ex:stretch-sharp}, and, in particular, the sharpness of the factor $\sqrt{\varepsilon}$ in Theorem~\ref{Thm:Spiral-Stretch-stab}.
\end{proof}

\begin{remark}
    The sharpness of the uniform convexity assumption on $\varphi$ in the statement of Theorem \ref{Thm:Spiral-Stretch-stab} can be proven analogously to Example \ref{ex:unif-convexity-assumption}, using the quasiconformal mappings defined in Example \ref{ex:stretch-sharp}. We leave the details to the interested reader.
\end{remark}

\section{Weak stability results  for strictly convex $\varphi$ }
\label{Sec:Weak}

In this section we relax the uniform convexity assumption imposed on the function $\varphi$ in the previous results and instead we only assume quadratic growth at infinity. As a consequence, the Taylor-expansion argument used for the proof of Proposition \ref{prop:lin-quant} is no longer available, and we need to construct a suitably chosen Young function adapted to $\varphi$ in order to quantify the deficit in an appropriate Orlicz norm. This way we obtain weaker -- but still optimal -- stability results. The necessary preliminaries regarding Young functions and the Orlicz-H\"older inequality can be found in the Appendix.

Similarly to the proof of Theorem \ref{Thm:Lin-Stretch} and \ref{Thm:Spiral-Stretch-stab}, we need the following proposition:

\begin{proposition} \label{prop:lin-quant-weak} 
Using the notations from Section \ref{secS2}, let $\varphi:[1, \infty) \to [1, \infty)$ be an increasing and strictly convex function with $\varphi(1) = 1$ and 
$$\liminf_{t\to\infty}\frac{\varphi(t)}{t^2}>0.$$ 
Then there exist $\varepsilon_0 >0$ and a function $\delta:(0,\varepsilon_0) \to(0,\infty)$ with $\lim_{\varepsilon\to 0^+}\delta(\varepsilon)=0$, such that  the following property holds. If 
\begin{equation}
\label{Inequality-Almost-Equality-weak}
\int_{Q_1} \varphi(K(z,f))\,d\mathcal{L}^2(z)\leq (1+\varepsilon)\int_{Q_1} \varphi(K(z,f^*))\,d\mathcal{L}^2(z)
\end{equation}
for some $0<\varepsilon < \varepsilon_0$ and $f \in \mathcal{F}$, 
then the mapping
$$
\Psi=f\circ (f^*)^{-1}: Q_2 \to \CC, \text{ where }  ~ ~  Q_2= f^{\ast}(Q_1)
$$ 
satisfies the estimate
\begin{equation}
\label{PsiConjugateDer-eq-weak}
\int_{Q_2} |\Psi_{\bar{w}}|(w) \,d\mathcal{L}^2(w) \ll \delta(\varepsilon).
\end{equation}
\end{proposition}

\begin{proof}
The argument follows a similar strategy to the proof of Proposition \ref{prop:lin-quant}. 
However, due to the lack of the uniform convexity assumption imposed  on $\varphi$, the proof requires a careful construction of an appropriate Young function tailored to $\varphi$, which allows us to apply the Orlicz-space results developed in the Appendix. 

\textit{Step 1.}
Let us denote $K(z,f^*) =:s\geq 1$, which is constant on $Q_1$. We aim to define a Young function that measures the convexity gap of $\varphi$ at $s$.

Let $\xi_0: \R \to [0, \infty)$  be the maximal even convex function such that 
$$
\varphi(t)\geq \varphi(s)+\varphi'_+(s)(t-s)+\xi_0(t-s),
$$
for all $t\geq 1$. 
In particular, the graph of $\xi_0$ is the boundary of the convex hull of epigraphs of
the convex function $\varphi_s:[1-s,\infty)\to [0,\infty)$,
$$
\varphi_s(\tau)=\varphi(s+\tau)-\varphi(s)-\varphi'_+(s)\tau,
$$
and the function $\tau\mapsto \varphi_s(-\tau)$, where $\tau\in(-\infty, s-1]$. Since  $\varphi$ is strictly convex, we have $\xi_0(0)=0$, $\xi_0(t)>0$ for $t\neq 0$, and $\lim_{|t|\to\infty}\xi_0(t)=\infty$. Furthermore, the condition
$\liminf_{t\to\infty}\frac{\varphi(t)}{t^2}>0$ implies that
$\liminf_{t\to\infty}\frac{\xi_0(t)}{t^2}>0$. 
Therefore, there exist  $\alpha>1$ and $\gamma>0$ such that
$$\xi_0(t) \geq \gamma t^2,  \quad \text{for all } t\geq \alpha.$$

Now, we define the Young function
$\xi:[0,\infty)\to [0,\infty)$,
$$
\xi(t)=\left\{
\begin{array}{ll}
\xi_0(t)&\mbox{ if }t\in[0,1]\\
\xi_0(1)+\xi'_{0,+}(1)(t-1)&\mbox{ if }t\in[1,\alpha]\\
\xi_0(1)+\xi'_{0,+}(1)(t-1)+\gamma(t-\alpha)^2& \mbox{ if }t\geq \alpha.
\end{array}\right.
$$
Readily, we have $\xi_0(t)\geq \xi(|t|)$, thus
\begin{equation}
  \label{phi-xi-Taylor}
\varphi(t)\geq  \varphi(s)+\varphi'_+(s)(t-s)+\xi(|t-s|).  
\end{equation}
We also have
\begin{align}
\label{xi-t2-infty}
\lim_{t\to\infty}\frac{\xi(t)}{t^2}=& ~ \gamma,\\
\label{xi-2t-infty}
\lim_{t\to\infty}\frac{\xi(2t)}{\xi(t)}=& ~ 4,
\end{align}
hence $\xi$ satisfies the doubling condition required in Lemma~\ref{doubling-makes-small}.

\textit{Step 2.}
Let $f \in \mathcal F$ be a homeomorphism satisfying \eqref{Inequality-Almost-Equality-weak}, and let $z \in Q_1$ be arbitrarily fixed.
Applying inequality \eqref{phi-xi-Taylor} to $t = K(z,f)$ and $s = K(z,f^*)$, we obtain that
$$
\varphi(K(z,f))-\varphi(K(z,f^*))\geq \varphi_+'(K(z,f^*))(K(z,f)-K(z,f^*))+\xi\big(|K(z,f)-K(z,f^*)|\big),
$$
where $\varphi_+'(K(z,f^*))\geq 0$.
Integrating over $Q_1$ and using the minimizing property \eqref{min_property_FengHuShen}, we obtain that 
$$ \int_{Q_1} \varphi(K(z,f)) \,d\mathcal{L}^2(z) - \int_{Q_1}\varphi(K(z,f^*)) \,d\mathcal{L}^2(z) \geq 
\int_{Q_1} \xi\big(|K(z,f)-K(z,f^*)|\big) \,d\mathcal{L}^2(z).$$
Using \eqref{Inequality-Almost-Equality-weak}, it follows that
$$ \int_{Q_1} \xi\big(|K(z,f)-K(z,f^*)|\big) \,d\mathcal{L}^2(z) \leq \varepsilon \cdot \int_{Q_1}\varphi(K(z,f^*)) \,d\mathcal{L}^2(z)  
\ll \varepsilon.$$
Combining this with Lemma~\ref{doubling-makes-small} yields that 
there exist $\varepsilon_0 >0$ and a function $\varepsilon \mapsto \delta_\xi(\varepsilon) >0$, $\varepsilon \in (0, \varepsilon_0)$ with $\lim_{\varepsilon\to 0^+}\delta_\xi(\varepsilon)=0$, such that
\begin{equation}
\label{K(z,f)-xinorm-small}
\left\|K(z,f)-K(z,f^*)\right\|_\xi \leq\delta_\xi(\varepsilon).
\end{equation}


Next, since $f(Q_1)$ is a fundamental domain for $L$, similarly to \eqref{intQ-square-diff} and \eqref{funddomain-fQ1}, we have 
\begin{equation*}
\label{intQ-square-diff-weak}
\int_{Q_1}(|f_z|-|f_{\bar{z}}|)^2\,d\mathcal{L}^2\leq 
\int_{Q_1}|f_z|^2-|f_{\bar{z}}|^2\,d\mathcal{L}^2 \ll 1.
\end{equation*}
Hence, for the Young function $\eta:[0,\infty)\to [0,\infty)$, $\eta(t)=t^2$, we have 
\begin{equation}
\label{L2-norm-weak}
\Big\||f_z|-|f_{\bar{z}}|\Big\|_\eta\ll 1.
\end{equation}
Furthermore, \eqref{xi-t2-infty} implies for the inverse function $\eta^{-1}(r) = \sqrt{r}, r \geq 0$ that
$$
\lim_{r\to\infty}\frac{\xi^{-1}(r)\sqrt{r}}{r}=\gamma^{-\frac12}.
$$
Hence, there exists  $r_0>0$ such that
\begin{equation*}
\frac{1}2 \gamma^{-\frac12}\cdot r\leq \xi^{-1}(r)\sqrt{r}\leq 2\gamma^{-\frac12}\cdot r\mbox{ \ \ \ for every }r\geq r_0.
\end{equation*}

In turn, one can construct a strictly increasing Young function $\aleph$ such that its inverse function dominates the product $\xi^{-1}(r)\sqrt r$. Indeed, let $\Phi:[0,\infty)\to[0,\infty)$ be the strictly increasing function given by
\[
\Phi(r)
=
\begin{cases}
\sqrt r, & 0\le r\le r_0,\\
2\gamma^{-1/2} r, & r\ge r_0,
\end{cases}
\]
for $r_0$ sufficiently large, 
and define $\aleph^{-1}$ to be the concave envelope of $\Phi$.
By construction, $\aleph^{-1}$  satisfies
\begin{equation*}
\xi^{-1}(r)\sqrt{r}\leq \aleph^{-1}(r), \mbox{ \ for all }r\geq 0.
\end{equation*}
Therefore, the inverse function $\aleph$
is a strictly increasing Young function with $\aleph_+'(0)=0$.

On the other hand, we have
\begin{equation*}
\Big|\left|f_{\bar{z}}\right|- \left|\mu^*f_{z}\right|\Big|=
c_0\left|K(z,f)-K(z,f^*)\right|
\Big|\left|f_{\bar{z}}\right|- \left|f_{z}\right|\Big|
\end{equation*}
for some constant $c_0>0$ depending only on $f^*$. Applying the Orlicz-H\"older inequality from Proposition~\ref{OrliczHolder-gen} for the above product, and using the estimates
\eqref{K(z,f)-xinorm-small} and
\eqref{L2-norm-weak} for
$\left\|K(z,f)-K(z,f^*)\right\|_\xi$ and $\Big\||f_z|-|f_{\bar{z}}|\Big\|_\eta$, we obtain that
 \begin{equation*}
\Big\||f_{\bar{z}}|-|\mu^*f_{z}|\Big\|_\aleph
\ll \delta_\xi(\varepsilon).
\end{equation*}
Finally, since $\aleph'(0)=0$, we may apply Lemma~\ref{Orlicz-Legendre}, which combined with the previous estimate  implies 
\begin{equation}
 \label{Lemma23-L1-norm-small}
 \int_{Q_1}
\Big||f_{\bar{z}}|-|\mu^*f_{z}|\Big|\,d\mathcal{L}^2
\ll\delta_\xi(\varepsilon),
\end{equation} 
which is the weaker analogue of Lemma \ref{abs-difference}.

\textit{Step 3}.
The rest of the proof follows the same arguments as the proof of Proposition \ref{prop:lin-quant}. In particular, we obtain the same estimate
$$
\int_{Q_2} |\Psi_{\bar{w}}|(w) \,d\mathcal{L}^2(w)\ll 
\int_{Q_1} \left|{\rm Re}\left(e^{i\alpha}f_{\bar{z}}\right)-
{\rm Re}\left(e^{i\alpha}\mu^*f_{z}\right)\right| \,d\mathcal{L}^2+ \int_{Q_1} \left|{\rm Im}\left(e^{i\alpha}f_{\bar{z}}\right)\right|+
\left|{\rm Im}\left(e^{i\alpha}\mu^*f_{z}\right)\right| \,d\mathcal{L}^2. 
$$
for the map $\Psi = f\circ(f^{\ast})^{-1}: Q_2 \to \mathbb C$, where  $Q_2 = f^{\ast}(Q_1)$. 
By \eqref{ImaginarySmall-eq} from Lemma \ref{ImaginarySmall}, we know that
the second integral on the right-hand side is of order at most $\sqrt{\varepsilon}$. For the first integral we use the triangle inequality, and then apply \eqref{RealPartSmall-eq} from Lemma \ref{ImaginarySmall} together with the estimate \eqref{Lemma23-L1-norm-small} from Step 2, obtaining
$$
\int_{Q_1}\Big|\Big({\rm Re}\left(e^{i\alpha}\mu^* f_z\right)-\left|e^{i\alpha}\mu^* f_z\right|\Big)-\Big({\rm Re}\left(e^{i\alpha}f_{\bar{z}}\right)-\left|e^{i\alpha}f_{\bar{z}}\right|\Big)\Big|\,d\mathcal{L}^2
+ \int_{Q_1}\Big|  \left|\mu^*f_{z}\right| -\left|f_{\bar{z}}\right|\Big|\,d\mathcal{L}^2\ll \varepsilon + \delta_\xi(\varepsilon).
$$
This completes the proof of Proposition~\ref{prop:lin-quant-weak}.
\end{proof}

Proposition \ref{prop:lin-quant-weak} yields the following weak stability results concerning the linear stretch map and the spilar stretch map in the same manner as Theorems \ref{Thm:Lin-Stretch} and \ref{Thm:Spiral-Stretch-stab} were proved based on Proposition \ref{prop:lin-quant}.

\begin{theorem} 
\label{Thm:Linear_Stretch-Stab-weak}
Let $\varphi:[1, \infty) \to [1, \infty)$ be an increasing and strictly convex function with $\varphi(1) = 1$ and $\liminf_{t\to\infty}\frac{\varphi(t)}{t^2}>0$. Then there exist $\varepsilon_0 >0$ and a function $\delta:(0,\varepsilon_0) \to(0,\infty)$ with $\lim_{\varepsilon\to 0^+}\delta(\varepsilon)=0$, such that it satisfies the following property. If
\begin{equation}
\label{Inequality-Almost-Equality-thm1.2-lin}
\int_{Q_1} \varphi(K(z,f))\,d\mathcal{L}^2(z)\leq (1+\varepsilon)\int_{Q_1} \varphi(K(z,f^*))\,d\mathcal{L}^2(z)
\end{equation} 
holds for some $0<\varepsilon < \varepsilon_0$ and some $f \in \mathcal{F}$ with $f = f^*$ on $\partial Q_1$, then
\begin{equation} \label{eq:lin-strech-stab-weak}
\int_{Q_1}|f-f^{\ast}|(z) \,d\mathcal{L}^2(z) \le \delta(\varepsilon).
    \end{equation}
\end{theorem}

\begin{theorem}
\label{phi-stricly-convex-problem}
Let $0<q<1$, $k>0$, $\theta \in [-\pi, \pi]$, and let
$\varphi: [1, \infty) \to [1, \infty)$ be an increasing, strictly convex function satisfying $\varphi(1)=1$ and 
$
\liminf_{t\to\infty}\frac{\varphi(t)}{t^2}>0.
$
Then there exist $\varepsilon_0 >0$ and a function $\delta:[0,\varepsilon_0]\to[0,\infty)$ depending on $q,k,\theta,\varphi$, and satisfying $\delta(\varepsilon)>0$ for all $\varepsilon\in(0,\varepsilon_0)$, and $\lim_{\varepsilon\to 0^+}\delta(\varepsilon)=0$, with the following property. If $g:A_1\to A_2$ is a quasiconformal mapping with
$g=g^{\ast}$ on $\partial A_1$, and
$$
\int_{A_1} \frac{\varphi(K(w,g))}{|w|^2}\,d\mathcal{L}^2(w)\leq(1+\varepsilon) 
\int_{A_1} \frac{\varphi(K(w,g^{\ast}))}{|w|^2}\,d\mathcal{L}^2(w),
$$
holds for small $0<\varepsilon< \varepsilon_0$, then
$$
\int_{A_1} |g-g^{\ast}| \,d\mathcal{L}^2\leq \delta(\varepsilon).
$$
\end{theorem}

\begin{remark} 
\begin{enumerate}
\item[1.] Observe that Theorem~\ref{phi-stricly-convex-problem} is more general than Theorem~\ref{Thm:Spiral_Stretch-weak-Stab}, since Theorem~\ref{Thm:Spiral_Stretch-weak-Stab} only uses the values $\varphi(t)$ for $t\leq K_0$; therefore, we may alter $\varphi$ on $[K_0,\infty)$ to ensure the condition $\liminf_{t\to\infty}\frac{\varphi(t)}{t^2}>0$. In particular, Theorem~\ref{phi-stricly-convex-problem} implies Theorem~\ref{Thm:Spiral_Stretch-weak-Stab}.
\item[2.] The condition $\liminf_{t\to\infty}\frac{\varphi(t)}{t^2}>0$ in Theorems \ref{Thm:Linear_Stretch-Stab-weak} and \ref{phi-stricly-convex-problem} is needed to ensure the existence of an appropriate Young function, which is necessary in the proof of Proposition \ref{prop:lin-quant-weak}.
\end{enumerate}
\end{remark}

Using the same ideas as in Example~\ref{ex:stretch-sharp}, we show the optimality of  Theorem~\ref{phi-stricly-convex-problem}.

\begin{example} \label{ex:strict-conv-error}
Let  $0<q<1$ and $k>1$. For the annuli 
$A_1=\{w\in\CC:q\leq |w|\leq 1\}$ and $A_2=\{w\in\CC:q^k\leq |w|\leq 1\}$, 
consider the radial stretch map,
$g^{\ast} : A_1 \to A_2, \ g^{\ast}(w) = w\cdot |w|^{k-1}$.
In addition, let $\varepsilon_0>0$, and $\delta:[0,\varepsilon_0]\to[0,\infty)$ be continuous such that $\delta(0)=0$ and $\delta(\varepsilon)>0$ for $\varepsilon>0$. Then we construct a strictly convex and increasing $\varphi:[1,\infty)\to[1,\infty)$ with $\varphi(1)=1$ (where $\varphi$ depends on $\delta(\varepsilon)$, $q$ and $k$) such that for any small $\varepsilon >0$, 
there exists a
 quasi-conformal $g=g^{(\varepsilon)}:A_1 \to A_2$ such that $g|_{\partial A_1} = g^{\ast}|_{\partial A_1}$ with the properties that 
\begin{equation} \label{eq:strict-conv-error-eps}
\int_{A_1} \frac{\varphi(K(w,g))}{|w|^2}\,d\mathcal{L}^2(w)\leq (1+\varepsilon) 
\int_{A_1} \frac{\varphi(K(w,g^{\ast}))}{|w|^2}\,d\mathcal{L}^2(w),
\end{equation}
and  
\begin{equation}
 \label{eq:strict-conv-error-delta-eps}
    \int_{A_1} |g-g^{\ast}|(w) d \mathcal{L}^2(w) \gg\delta(\varepsilon)
\end{equation}
where the implied constant factor in $\gg$ depends on $k$ and $q$.
\end{example}

\begin{proof}
To construct the suitable $\varphi$ in \eqref{ex:strict-conv-error}, let $\tau\in(0,\varepsilon_0)$ be the minimal value with the property $\delta(\tau)=\max_{s \in [0, \varepsilon_0]} \delta(s)$. By considering the upper envelope of the convex hull of the set 
$$\{(x,y)\in\R^2:x\in[0,\tau]\;\&\;0\leq y\leq\delta(x)\}\cup \{(x,y)\in\R^2:x\in[\tau,\varepsilon_0]\;\&\;0\leq y\leq\delta(\tau)\},$$ 
we may assume that $\delta$ is concave and increasing. Next, by replacing $\delta$ with $\delta(\varepsilon)+\sqrt{\varepsilon}$, we may assume that $\delta$ is strictly concave, increasing on $[0,\varepsilon_0]$, and satisfies $\delta(0)=0$ and $\delta'(0)=\infty$. Therefore, there exists an even strictly convex function $\psi$ on $[-\delta(\varepsilon_0),\delta(\varepsilon_0)]$ such that $\psi(-\delta(\varepsilon))=\psi(\delta(\varepsilon))=\varepsilon$ for $\varepsilon\in[0,\delta(\varepsilon_0)]$ and $\psi(0)=\psi'(0)=0$, and hence there exist an $\eta>0$ and a strictly convex function $\varphi:[1,\infty)\to[1,\infty)$ such that $\varphi(k)=\varphi'(k)=k$, and $\varphi(k+t)=k+kt+\psi(t)$ if $|t|\leq \eta$.

Analogously to Example~\ref{ex:stretch-sharp}, we have $g^{\ast}(w)= w^{\frac{k+1}{2}}\cdot \bar{w}^{\frac{k-1}{2}}$. For every $\varepsilon\in[0,\eta]$,  we define
$$
g(w)=g^{(\varepsilon)}(w)=
\left\{\begin{array}{rl}
q^{\delta(\varepsilon)} \cdot w\cdot |w|^{(k-1 -\delta(\varepsilon))} & \mbox{ if } \ |w| \in [q,q^{\frac{1}{2}}], \\ 
w\cdot |w|^{(k-1+\delta(\varepsilon))} 
&\mbox{ if } \ |w| \in [q^{\frac{1}{2}}, 1],
\end{array}\right.
$$
hence \eqref{eq:strict-conv-error-delta-eps} follows analogously to Example~\ref{ex:stretch-sharp}.
In addition, $K(w, g^{\ast})=k$ and
$$
K(w, g)=
\left\{\begin{array}{rl}
k-\delta(\varepsilon) & \mbox{ \ if } \ |w| \in [q,q^{\frac{1}{2}}), \\ 
k+\delta(\varepsilon)
&\mbox{ \ if } \ |w| \in (q^{\frac{1}{2}}, 1],
\end{array}\right.
$$
thus the definition of $\varphi$ implies that
$$
\varphi(K(w, g))=
\left\{\begin{array}{rl}
k-k\delta(\varepsilon) +\varepsilon& \mbox{ \ if } \ |w| \in [q,q^{\frac{1}{2}}), \\ 
k+k\delta(\varepsilon) +\varepsilon&\mbox{ \ if } \ |w| \in (q^{\frac{1}{2}}, 1].
\end{array}\right.
$$
In turn, we conclude \eqref{eq:strict-conv-error-eps} and the statement of Example~\ref{ex:strict-conv-error} using the arguments from Example~\ref{ex:stretch-sharp}.
\end{proof}

\section{Extensions to $L^p$ stability and further remarks} 
\label{secS5}

A natural question  is whether our stability results can be extended to more general $L^p$ estimates for certain $p\ge1$. Regarding this direction, we can make two straightforward observations.

First, using properties of the Cauchy transform, the results of Theorems \ref{Thm:Lin-Stretch} and \ref{Thm:Spiral-Stretch-stab} can be also formulated for $L^p$-norms when $1\leq p <2$, as long as we assume equality on the boundary, namely:

\begin{theorem} \label{Thm:Lin-Stretch_2>p>1}
Let $1\leq p <2$. Under the assumptions of Theorem \ref{Thm:Lin-Stretch}, if $f \in \mathcal{F}$ satisfies $f=f^{\ast}$ on $\partial Q_1$ and
\begin{equation} \label{almost_eq_Thm5.1}
\int_{Q_1} \varphi(K(z,f))\,d\mathcal{L}^2(z)\leq (1+\varepsilon)\int_{Q_1} \varphi(K(z,f^*))\,d\mathcal{L}^2(z)
\end{equation}
for some $\varepsilon>0$ sufficiently small, then
\begin{equation*} 
\left(\int_{Q_1}|f-f^{\ast}|^p(z) \,d\mathcal{L}^2(z)\right)^{1/p} \ll \sqrt{\varepsilon}.
\end{equation*} 
\end{theorem}

\begin{proof}
Analogously to the proof of Theorem \ref{Thm:Lin-Stretch}, we set $Q_2 = f^{\ast}(Q_1)$ and $\Psi = f\circ (f^{\ast})^{-1} \in W^{1,1}(Q_2)$.
By the change of variables $\xi = f^{\ast}(z)$, it follows that  
$$ \int_{Q_1}|f-f^{\ast}|^p(z) \, d\mathcal{L}^2(z) = \int_{Q_2} |\Psi (\xi)- \xi|^p J(\xi, (f^{\ast})^{-1}) \, d\mathcal{L}^{2}(\xi) \ll \int_{Q_2} |\Psi (\xi)- \xi|^p  \, d\mathcal{L}^{2}(\xi). $$
Using the Cauchy-Pompeiu formula, the Cauchy integral formula and the assumption $f=f^{\ast}$ on $\partial Q_1$,
we obtain that
\begin{equation*} 
\int_{Q_2} |\Psi (\xi)- \xi|^p  \, d\mathcal{L}^{2}(\xi) =  \int_{Q_2} \left|\frac{1}{\pi }\int_{Q_2}\frac{\Psi_{\bar{w}}(w)}{w-\xi}\, d\mathcal{L}^2(w) \right|^p  \, d\mathcal{L}^2(\xi).
\end{equation*}
Notice that 
$$
(\mathcal C \Psi_{\bar{w}})(\xi) =
\frac{1}{\pi}\int_{Q_2} 
\frac{\Psi_{\bar{w}}(w)}{\xi-w}\, d\mathcal{L}^2(w)
$$
is the Cauchy transform of $\Psi_{\bar{w}} \in L^1(Q_2)$ when $\Psi_{\bar{w}}$ is extended to $0$ outside $Q_2$.
Therefore, the problem reduces to estimating the $L^p$-norm of the Cauchy transform $\mathcal C \Psi_{\bar{w}}$. In particular, since $Q_2$ is bounded, we have that for every $1\leq p<2$,
$$\|\mathcal{C} \Psi_{\bar{w}} \|_{L^p(Q_2)} \leq C \|\Psi_{\bar{w}}\|_{L^1(Q_2)}, $$
where $C$ depends only on $p$ and on the set $Q_2$ (i.e., on the numbers $\ell, k$ and $n$), see \cite[Theorem 4.3.6.]{AIM09}.

Therefore, we obtain that
\begin{equation*}  
\left(\int_{Q_2} |\Psi (\xi)- \xi|^p  \, d\mathcal{L}^{2}(\xi)\right)^{1/p} \ll \int_{Q_2} 
|\Psi_{\bar{w}}(w)|\, d\mathcal{L}^2(w) \ll \sqrt{\varepsilon},
\end{equation*}
where we applied Proposition \ref{prop:lin-quant} for the last estimate. 
\end{proof}

An analogous extension holds for Theorem \ref{Thm:Spiral-Stretch-stab}, namely:

\begin{theorem} \label{Thm:Spiral-Stretch_2>p>1}
Let $1\leq p <2$. Under the assumptions of Theorem \ref{Thm:Spiral-Stretch-stab}, if $g \in W^{1,1}(A_1, A_2)$ is an orientation preserving homeomorphism  with finite distortion satisfying $g=g^{\ast}$ on $\partial A_1$  and
\begin{equation}\label{almost_eq_Thm5.2}
\int_{A_1} \frac{\varphi(K(w,g))}{|w|^2}\,d\mathcal{L}^2(w)\leq(1+\varepsilon) \int_{A_1} \frac{\varphi(K(w,g^{\ast}))}{|w|^2}\,d\mathcal{L}^2(w)
\end{equation}
for some $\varepsilon>0$ sufficiently small, then
\begin{equation*} 
\left(\int_{A_1}|g-g^{\ast}|^p(z) \,d\mathcal{L}^2(z)\right)^{1/p} \ll \sqrt{\varepsilon}.
\end{equation*} 
\end{theorem}

On the other hand, if $f \in \mathcal{F}$ satisfies the boundary condition $f=f^{\ast}$ on $\partial Q_1$, then there exists a constant $C_1 = C_1(k, \ell, n)$ such that
\begin{equation}\label{trivial_eq_quadrilaterals}
    |f-f^{\ast}|(z) \leq C_1 ,\quad \text{ for all } z \in Q_1.
\end{equation}
Using this and the  $L^1$ estimate of Theorem \ref{Thm:Lin-Stretch}, one can obtain the following  $L^p$ estimate for general $p \geq 1$.

\begin{theorem}\label{Thm:Linear-Stretch_p>1_weaker}
    Let $p \geq 1$. Under the assumptions of Theorem \ref{Thm:Lin-Stretch}, if $f \in \mathcal{F}$ satisfies $f=f^{\ast}$ on $\partial Q_1$ and relation \eqref{almost_eq_Thm5.1}, then
\begin{equation*} 
\left(\int_{Q_1}|f-f^{\ast}|^p(z) \,d\mathcal{L}^2(z)\right)^\frac{1}{p} \ll \varepsilon^\frac{1}{2p}.
\end{equation*} 
\end{theorem}

\begin{proof}
First, by Theorem \ref{Thm:Lin-Stretch},  there exists a constant $C_2>0$ depending on $k, \ell, n$ and $\varphi$ such that
    $$ \int_{Q_1} |f-f^{\ast}|(z) d\mathcal{L}^2(z) \leq C_2 \sqrt{\varepsilon}.$$
Let us consider the set 
$$Q = \{ z\in Q_1:|f-f^{\ast}|(z) >1 \} \subset Q_1.$$
Then, we have that 
$$ \mathcal{L}^2(Q) = \int_Q 1 d\mathcal{L}^2 \leq \int_Q |f-f^{\ast}|(z) d\mathcal{L}^2(z) \leq \int_{Q_1} |f-f^{\ast}|(z) d\mathcal{L}^2(z) \leq C_2 \sqrt{\varepsilon}.$$
Using the above estimate and inequality \eqref{trivial_eq_quadrilaterals}, it follows that  
\begin{align*}
\int_{Q_1} |f-f^{\ast}|^p(z) d\mathcal{L}^2(z) 
& \leq \int_{Q} |f-f^{\ast}|^p(z) d\mathcal{L}^2(z) +\int_{Q_1\setminus Q} |f-f^{\ast}|^p(z) d\mathcal{L}^2(z)  \\
& \leq C_1^p \mathcal{L}^2(Q) + \int_{Q_1\setminus Q} |f-f^{\ast}|(z) d\mathcal{L}^2(z) \\
&\leq (C_1^p +1) C_2 \sqrt{\varepsilon},
\end{align*}
which concludes the proof.    
\end{proof}

Consequently, we have an analogous estimate in the case of the annuli.

\begin{theorem}\label{Thm:Spiral-Stretch_p>1_weaker}
    Let $p \geq 1$. Under the assumptions of Theorem \ref{Thm:Spiral-Stretch-stab}, if $g \in W^{1,1}(A_1, A_2)$ is an orientation preserving homeomorphism  with finite distortion satisfying $g=g^{\ast}$ on $\partial A_1$  and relation \eqref{almost_eq_Thm5.2}, then
\begin{equation*} 
\left(\int_{A_1}|g-g^{\ast}|^p(z) \,d\mathcal{L}^2(z)\right)^\frac{1}{p} \ll \varepsilon^\frac{1}{2p}.
\end{equation*} 
\end{theorem}

We close our discussion with the following questions which arise naturally  from the previous results:
\begin{enumerate}
    \item[1.] 
    For which values of $p$ are the stability estimates of this section sharp? For example, when $p \in [1,2)$, we can observe that Theorems \ref{Thm:Lin-Stretch_2>p>1} and \ref{Thm:Spiral-Stretch_2>p>1} yield the same $\sqrt{\varepsilon}$ rate as the optimal $L^1$ estimates of Theorems \ref{Thm:Lin-Stretch} and \ref{Thm:Spiral-Stretch-stab}. On the other hand,  Theorems \ref{Thm:Linear-Stretch_p>1_weaker} and \ref{Thm:Spiral-Stretch_p>1_weaker} give only a weaker rate  $\varepsilon^{\frac{1}{2p}}$.    
    \item[2.] Is it possible to improve the above $L^p$ estimates by assuming stronger regularity on the admissible mappings (e.g., higher Sobolev regularity)?
    \item[3.] Considering the case $p=\infty$, can we obtain similar estimates of the form $\|f-f^\ast\|_\infty \leq \delta(\varepsilon)$, with  $\lim_{\varepsilon\to 0^+}\delta(\varepsilon)=0$?
    \item[4.]
    In the case of the Gr\"otzsch minimization problem, the admissible function class $\mathcal F$ fixes the four vertices of the quadrilateral $Q_2$. Is it possible to obtain similar stability estimates \textit{without} assuming equality (or almost equality) on the boundary $\partial Q_1$ for $f$ and $f^\ast$? In particular, does the almost-minimality assumption \eqref{almost_eq_Thm5.1} on $f \in \mathcal F$ already imply a boundary condition of the form 
    $\int_{\partial Q_1}|f-f^{\ast}|(z) \,dz \ll \delta(\varepsilon)$ for some function $\delta$ with  $\lim_{\varepsilon\to 0^+}\delta(\varepsilon)=0$? If this is not the case, then can we  provide a counterexample to indicate the necessity of the boundary condition \eqref{eq:boundary-cond}?
\end{enumerate}


\section*{Appendix}
\label{secYoungOrlicz} 

This section gives a brief summary on  Young functions, the Orlicz-H\"older inquality, the bounded doubling condition and complementary Young functions. Furthermore, we deduce some preliminary results which are necessary for our proofs from Section \ref{Sec:Weak}. For further details see O'Neil \cite{ONeil}, Ifronika, Masta, Nur, Gunawan \cite{Gunawan}, and  Rao \cite{Rao}.

A function $\xi:[0,\infty)\to[0,\infty)$ is called a Young function if
$\xi$ is convex, $\xi(0)=0$
and $\lim_{t\to\infty}\xi(t)=\infty$.
Observe that a Young function $\xi$ is strictly increasing if and only if $\xi(t)>0$ for all $t>0$.

Let $(X,\mu)$ be a Borel measure space and $\xi$ a Young function. The associated Orlicz space $L_\xi(X)$ consists of all $\mu$-measurable functions $h:X\to\R$ such that
$$
\int_{X}\xi(a|h(x)|)\,d\mu(x)<\infty
$$
for some $a>0$. For $h\in L_\xi(X)$,
the Orlicz norm (or Luxemberg norm) of $h$ is defined as
$$
\|h\|_\xi=\inf\left\{ b>0: \int_{X}\xi\left(\frac{|h(x)|}b\right)\,d\mu(x)\leq  1\right\}.
$$
With this norm $\|\cdot \|_\xi:L_\xi(X) \to [0, \infty) $, $L_\xi(X)$ forms a Banach space.
Observe that if $\xi(t)=t^p$ for some $1 < p< \infty$, then $L_\xi(X)=L^p(X)$ and $\|\cdot\|_\xi$ coincides with the classical $L^p$ norm. Therefore, the Orlicz space $L_\xi(X)$ is in fact a
generalization of the Lebesgue space $L^p(X)$ .

For a strictly increasing Young functions $\xi$, let $\xi^{-1}:[0,\infty)\to [0,\infty)$ be its inverse function. Then, $\xi^{-1}$ is also strictly increasing, concave, and satisfies $\xi^{-1}(0)=0$, $\lim_{r\to\infty}\xi^{-1}(r)=\infty$, and hence 
\begin{equation}
\label{Young-inverse-sublinear}
\limsup_{r\to\infty}\frac{\xi^{-1}(r)}{r}<\infty.
\end{equation}

We recall the following generalization of H\"older's inequality in Orlicz spaces due to O'Neil \cite{ONeil}.

\begin{proposition}[Orlicz-H\"older inequality, \cite{ONeil}]
\label{OrliczHolder-gen}
If $~ \xi,\eta,\aleph:[0,\infty)\to[0,\infty)$ are strictly increasing Young functions such that
\begin{equation}
\label{OrliczHolder-cond-prod}
\xi^{-1}(r)\eta^{-1}(r)\leq \aleph^{-1}(r) ~ \text{ for all } ~ r\geq 0,
\end{equation}
 then 
for any $h\in L_\xi(X)$ and $m \in L_{\eta}(X)$, we have
\begin{equation}
\label{OrliczHolder-gen-eq}
\|hm\|_{\aleph}\leq
2 \|h\|_{\xi}\cdot\|m\|_{\eta}.
\end{equation}
\end{proposition}
\begin{remark}
 If  $\xi(t)=t^p$ and $\eta(t)=t^q$ for some $p,q > 1$ with $\frac1p+\frac1q=1$, then the above result yields  a slightly weaker form of the classical H\"older inequality.  
\end{remark} 

In applications, it is often expected that when we consider a sequence of functions $\{h_n\}$ such that $\int_{X}\xi(|h_n|)\,d\mu$ tends to zero, then $\|h_n\|_\xi$ also tends to zero when $n \to \infty$. While this property clearly holds for the $L_p$ norms, it may not hold for a general Young function $\xi$. Therefore, we introduce the following condition. 

We say that a Young function $\xi$ satisfies the \textit{bounded doubling} condition if 
\begin{equation}
\label{Boundeddoubling-def}
\limsup_{t\to\infty}\frac{\xi(2t)}{\xi(t)}<\infty.
\end{equation}
For instance, this condition is satisfied by the functions $\xi(t) = t^p$, $p>1$,  generating the $L_p$ norms.

In the following we present some results considering Young functions satisfying the bounded doubling condition.

\begin{lemma}
\label{Boundeddoubling-prop}
If $\xi$ is a Young function  satisfying the bounded doubling condition, then for any $a>0$ with $\xi(a)>0$ and $b\in(0,1)$, there exists a constant $C_{a,b}>1$ depending only on $a$, $b$ and $\xi$, such that
$$
\xi(t/b)\leq C_{a,b}\,\xi(t)   , ~ \text{ for all } ~ t\geq a.
$$
\end{lemma}
\begin{proof} First we show the existence for $b = \frac{1}{2}$. According to \eqref{Boundeddoubling-def}, 
there exist constants $d>0$  and  $C>1$ such that $\xi(d)>0$ and
$$
\xi(2t)\leq C\xi(t),   \text{   for all } t\geq d.
$$
This settles the existence of $C_{a,\frac12}$ when $a\geq d$. Next, assume that $0<a<d$ with $\xi(a)>0$. For every $t\in[a,d]$,  the monotonicity of the right derivative  $\xi'_+$ of $\xi$ and the relation $\frac{\xi(t)}{t}\geq \frac{\xi(a)}{a}$ imply that
$$
\xi(2t)-\xi(t)\leq \xi'_+(2d)\cdot t\leq
\frac{a\, \xi'_+(2d)}{\xi(a)}\cdot \xi(t).
$$
By combining this with the estimate for $t\geq d$ and choosing  $C_{a,\frac12}:=\max\left\{C, 1+\frac{a \xi'_+(2d)}{\xi(a)}\right\}$, we obtain  that $\xi(2t) \leq C_{a, \frac{1}{2}}\xi(t)$ for all $t\geq a$, which proves the claim for $b = \frac{1}{2}$.

Now, for general $b\in(0,1)$,  let $n$ be the smallest positive integer with the property that 
$\frac{1}{b}\leq 2^n$. For all $t \geq a$, applying $n$ times the inequality $\xi(2t) \leq C_{a, \frac{1}{2}}\xi(t)$, we obtain that 
$$ \xi \left(t/b \right)\leq \xi(2^n t) \leq  \left( C_{a, \frac{1}{2}}\right)^n\xi(t).$$
To finish the proof,  notice that $n \leq 1+ \log_2(\frac{1}{b})$. Therefore, 
the statement of the lemma holds for the choice of 
\begin{equation}
\label{eq:doubling-constant} 
C_{a,b} := \left(C_{a, \frac{1}{2}}\right)^{1+ \log_2(\frac{1}{b})}
\leq \left(C \frac{a}{\xi(a)}\right)^{1+ \log_2(\frac{1}{b})},
\end{equation}
where $C_{a,\frac12}=\max\left\{C, 1+\frac{a \xi'_+(2d)}{\xi(a)}\right\}$ and the constants $C,d$ depend only on the function $\xi$.

\end{proof}

In turn, Lemma~\ref{Boundeddoubling-prop} implies the following crucial property of Young functions satisfying the bounded doubling condition. 

\begin{lemma}
\label{doubling-makes-small}
Let $(X, \mu)$ is a Borel measure space with $\mu(X)<\infty$, and $\xi$ is a Young function satisfying the bounded doubling condition. Then, for any $\varepsilon\in(0,1)$ small enough there exists $\delta_\xi(\varepsilon)>0$ with $\lim_{\varepsilon\to 0^+}\delta_\xi(\varepsilon)=0$, such that  any $\mu$-measurable function $h:X\to \R$ satisfying
$\int_X\xi(|h(x)|)\,d\mu(x)\leq \varepsilon$
also satisfies 
$\|h\|_\xi\leq \delta_\xi(\varepsilon)$.
\end{lemma}

\begin{proof}
The proof is based on the explicit estimate of the constant $C_{a,b}$  given by \eqref{eq:doubling-constant}. More precisely, we apply Lemma~\ref{Boundeddoubling-prop} with the choice of a number $0<a<1/2$ (later to be fixed) and $b := \sqrt{a}$, to obtain that there exists a constant 
$$C_{a, \sqrt{a}}\leq  \left(C \frac{a}{\xi(a)}\right)^{1+ \log_2(\frac{1}{\sqrt{a}})}$$
such that 
\begin{equation} \label{estimate_from_Lemma6.1}
  \xi(t/\sqrt{a})\leq C_{a,\sqrt{a}}\,\xi(t)   , ~ \text{ for all } ~ t\geq a,  
\end{equation}
where $C$ is the constant from Lemma~\ref{Boundeddoubling-prop} depending only on $\xi$.

Let $h:X\to \R$ be a $\mu$-measurable function such that
$\int_X\xi(|h|)\,d\mu \leq \varepsilon$. 
Applying \eqref{estimate_from_Lemma6.1}, we obtain that
\begin{align*}
 \int_X\xi\left(\frac{|h|}{\sqrt{a}}\right) d\mu & = \int_{|h| \leq a} \xi\left(\frac{|h|}{\sqrt{a}}\right) d\mu + \int_{|h|> a} \xi\left(\frac{|h|}{\sqrt{a}}\right) d\mu \\ & \leq \mu(X)\xi(\sqrt{a}) + C_{a, \sqrt{a}}\int_X \xi(|h|)d\mu.
\end{align*}
Combining the above estimates, we obtain that
\begin{equation} \label{eq:xi-norm}
  \int_X\xi\left(\frac{|h|}{\sqrt{a}}\right) d\mu  \leq \mu(X)\xi(\sqrt{a}) + \left(C \frac{a}{\xi(a)}\right)^{1+ \log_2(\frac{1}{\sqrt{a}})}   \int_X \xi(|h|)d\mu.
\end{equation}
Next, define the function $\psi: (0, \frac{1}{2}) \to (\psi(\frac{1}{2}), \infty)$ by  
 $$ \psi(a):= \left(C \frac{a}{\xi(a)}\right)^{1+ \log_2(\frac{1}{\sqrt{a}})}. $$
Observe that $\psi$ is continuous, strictly decreasing, and $\lim_{a \to 0^+} \psi(a) = \infty$. Therefore, it admits a continuous inverse $\psi^{-1}: (\psi(\frac{1}{2}), \infty) \to (0, \frac{1}{2})$ which is also strictly decreasing. 

Thus, for any $\varepsilon \in (0,1)$ small enough such that $\frac{1}{\sqrt{\varepsilon}}> \psi(\frac{1}{2})$,  
 we may choose 
 $$a=a(\varepsilon) := \psi^{-1}\big(\frac{1}{\sqrt{\varepsilon}}\big) ~ \in ~ \big(0, \frac{1}{2}\big).$$ 
 Applying \eqref{eq:xi-norm} for this value of $a=a(\varepsilon)$ and using that $\mu(X)< \infty$, we obtain for any $\varepsilon \in (0,1)$ sufficiently small that 
 $$\int_X\xi\left(\frac{|h|}{\sqrt{a}}\right) d\mu  \leq \mu(X)\xi(\sqrt{a})+ \sqrt{\varepsilon}< 1 .$$
 This implies that
 $$ ||h||_{\xi} \leq \sqrt{a} = \sqrt{\psi^{-1}(\frac{1}{\sqrt{\varepsilon}}}), $$
 which proves the statement for the choice of 
 $$ \delta_{\xi}(\varepsilon) := \sqrt{\psi^{-1}(\frac{1}{\sqrt{\varepsilon}}}).$$ 
\end{proof}

It is not difficult to construct examples showing that the bounded doubling condition 
is essential for Lemma~\ref{doubling-makes-small}.

Next, let $\aleph:[0,\infty)\to[0,\infty)$ be a strictly increasing Young function with $\aleph'_+(0)=0$, and define  $d=\lim_{t\to\infty}\frac{\aleph(t)}t=\sup_{t>0}\aleph'_+(t)\in(0,\infty]$. 
The complementary Young function $\aleph^*:[0,d)\to[0,\infty)$ associated to $\aleph$ is defined as the Legendre transform  
\begin{equation}
\label{Legendre-def}
\aleph^*(s)=\sup_{t\geq 0}\{st-\aleph(t)\},
\end{equation}
for all $s\in [0,d)$. Then $\aleph^*$ is a strictly increasing, convex function with $\aleph^*(0)=0$. Note that definition \eqref{Legendre-def} can be extended to a convex function $\aleph^*:[0,\infty)\to[0,\infty]$, where
$\aleph^*(s)=\infty$ for all $s>d$; although, this extended definition will not be used here. Clearly, \eqref{Legendre-def} also yields the Young inequality 
\begin{equation}
\label{Young-ineq}
ts\leq \aleph(t)+\aleph^*(s), \quad \text{ for all } t\geq 0 \text{ and } s\in[0,d).
\end{equation}

Finally, we deduce the following result, which is a special case of the Orlicz-H\"older inequality for the complementary Young function.

\begin{lemma}[Orlicz norm and Legendre transform]
\label{Orlicz-Legendre}
Let $(X, \mu)$ is a Borel measure space with $\mu(X)<\infty$,  $\aleph:[0,\infty)\to[0,\infty)$ be a strictly increasing Young function with $\aleph'_+(0)=0$, and let  $d=\lim_{t\to\infty}\frac{\aleph(t)}t\in(0,\infty]$.
If $h\in L_\aleph(X)$ and $s\in(0,d)$, then
\begin{equation}
\label{OrliczHolder-Legendre-eq}
\int_{X}|h|\,d\mu\leq
\frac{1+ \aleph^*(s)\cdot\mu(X)}{s}\cdot \|h\|_{\aleph}.
\end{equation}
\end{lemma}

\begin{proof} Let $\lambda>1$ be arbitrarily fixed. For any $z\in X$ and $s\in(0,d)$, \eqref{Young-ineq} yields that
$$
\frac{|h(z)|}{\lambda\|h\|_{\aleph}}\cdot s
\leq
\aleph\left(\frac{|h(z)|}{\lambda\|h\|_{\aleph}}\right)+\aleph^*(s).
$$
Integrating over $X$, we obtain that
$$
\frac{s}{\lambda\|h\|_{\aleph}}\int_X|h|\,d\mu\leq  \int_X \aleph\left(\frac{|h|}{\lambda\|h\|_{\aleph}}\right)\,d\mu + \aleph^*(s)\cdot\mu(X)\leq  1+ \aleph^*(s)\cdot\mu(X).
$$
Letting $\lambda \to 1$ concludes the proof.
\end{proof}

\end{document}